\documentclass[12pt]{article}

 \usepackage[margin=1.1in]{geometry}
  \usepackage{amsmath,amssymb,amsfonts,graphicx,amsthm,nicefrac,mathtools,bm}
\usepackage{courier}
\usepackage[round]{natbib}
\usepackage{textcomp}

\usepackage{sectsty}
\usepackage{hyperref}
\hypersetup{colorlinks,
linkcolor=black,
citecolor=blue,
urlcolor=black
}
\usepackage{bbm}
\usepackage{float}
\usepackage{subcaption}
\usepackage{enumitem}
\usepackage[boxruled,linesnumbered]{algorithm2e}

\newtheorem{theorem}{Theorem}
\newtheorem{lemma}{Lemma}
\newtheorem{proposition}{Proposition}
\newtheorem{corollary}{Corollary}
\theoremstyle{remark}

\theoremstyle{definition}
\newtheorem{definition}{Definition}

\renewcommand{\vec}[1]{\boldsymbol{{#1}}}
\DeclareMathOperator*{\argmax}{arg\,max}
\DeclareMathOperator*{\argmin}{arg\,min}

\newcommand{\PP}{\mathbb{P}}
\newcommand{\EE}{\mathbb{E}}
\newcommand{\RR}{\mathbb{R}}

\newcommand{\bZ}{\vec{Z}}

\newcommand{\bY}{\vec{Y}}
\newcommand{\by}{\vec{y}}
\newcommand{\bX}{\vec{X}}
\newcommand{\bx}{\vec{x}}
\newcommand{\bt}{\vec{t}}

\newcommand{\btheta}{\boldsymbol{\theta}}

\newcommand{\bTheta}{\boldsymbol{\Theta}}

\newcommand{\calN}{\mathcal{N}}
\newcommand{\calA}{\mathcal{A}}

\newcommand{\calY}{\mathcal{Y}}

\newcommand{\ba}{\vec{a}}
\newcommand{\calD}{\mathcal{D}}

\newcommand{\thetatil}{\tilde{\theta}}
\newcommand{\bthetatil}{\tilde{\btheta}}
\newcommand{\bthetatilprime}{\tilde{\btheta}'}

\newcommand{\Ytil}{\tilde{Y}}
\newcommand{\bYtil}{\tilde{\bY}}
\newcommand{\bYtilprime}{\tilde{\boldsymbol{Y}}'}
\newcommand{\bytil}{\tilde{\by}}
\newcommand{\deltabayes}{\tilde{\delta}^*}

\newcommand{\bzero}{\boldsymbol{0}}
\newcommand{\htil}{\tilde{h}}

\newcommand{\bYnotj}{\boldsymbol{Y}_{-j}}
\newcommand{\gjzerostar}{g_{j0}^*}
\newcommand{\gjzerostartilde}{\tilde{g}_{j0}^*}
\newcommand{\gjonestartilde}{\tilde{g}_{j1}^*}
\newcommand{\bfone}{\boldsymbol{1}}

\DeclareMathOperator*{\minimize}{minimize}

\title{On the Minimum Attainable Risk in Permutation Invariant Problems}

\author{
Asaf Weinstein\\ [1ex]
{\it \normalsize Department of Statistics, The Hebrew University of Jerusalem}\\
}

\date{}

\providecommand{\keywords}[1] 
{ 
\small 
\textit{Keywords: } #1 
}

\begin{document}
\maketitle	

	\begin{abstract}

We introduce a broad class of permutation invariant problems by extending the standard decision theoretic definition to allow also {\em selective inference} tasks, where the target is specified only after seeing the data. For any such problem, the minimizer of the risk at $\btheta$ among all permutation invariant (equivariant) procedures is shown to be the Bayes rule that posits a uniform prior over all permutations of $\btheta$. This gives an explicit form of the greatest lower bound on the risk of any sensible procedure in a wide range of problems. 
From a practical perspective, approximations to the exact bound are required because of its computational cost. 
In a specific example of estimating the parameter of a selected population, we prove that our bound coincides asymptotically with the computationally tractable bound attained by the Bayes rule which replaces the 
uniform prior on all permutations of $\btheta$ by the i.i.d.~prior with the same marginals. This generalizes results previously known only for the very special case of compound decision problems. The possibility of asymptotically attaining the latter bound by an empirical Bayes rule is discussed.

	\end{abstract}
	

\section{Introduction}\label{sec:intro}

Consider a general statistical problem in the standard decision theoretic framework, entailing a parametric model $Y\sim f(\cdot;\theta)$, the fixed unknown parameter $\theta\in \Theta$ indexing a distribution on some sample space $\calY$, and a loss function $L(\theta, a)$ associating a nonnegative penalty with each possible value of the parameter and every action $a\in \calA$. 
The statistical problem is  to choose a decision rule $\delta:\calY\to \calA$ such that the risk $R(\theta, \delta) = \EE_{\theta}L(\theta, \delta(Y))$ is small; because $\theta$ is unknown, ideally we want $\delta$ to make the risk small for all  $\theta\in \Theta$. 
There is, however, an inherent difficulty in the frequentist paradigm, as the minimizer
\begin{equation}\label{eq:oracle-triv}
\delta^* = \argmin_{\delta}R(\theta,\delta)
\end{equation}
will generally depend on $\theta$, i.e., there is no single rule $\delta$ that is optimal for all values $\theta\in \Theta$ at the same time. 
The usual way to circumvent this difficulty is to resort to minimaxity or Bayes criteria, but minimizing worst-case or average risk (with respect to some pre-chosen weight function on $\Theta$) promises little about performance at the true value $\theta$, which should ultimately be the only concern. 

Still, for some {\em special} classes of statistical problems, solutions with much stronger (asymptotic) guarantees compared to minimaxity or Bayes in fact exist. 
This is the case in {compound decision} problems, a class of statistical problems defined and studied by Robbins starting in the early 1950s. 



\subsection{Compound decision problems}\label{sec:cd}
A compound decision problem consists of $n$ independent sub-problems of the same formal structure, to be solved simultaneously under the average componentwise loss. 
Formally, the statistician observes independent variables taking values in some sample space $\calY$, 
\begin{equation}
\label{eq:model-cd}
Y_i\stackrel{ind}{\sim} f(\cdot; \theta_i), \qquad i=1,...,n, 
\end{equation}
and must make a decision of some kind $a_i\in \calA$ regarding each of the $n$ unknown parameters $\theta_i\in \Theta$. 
Denoting $\bY = (Y_1,...,Y_n), \btheta = (\theta_1,..., \theta_n)$, the loss for action $\ba = (a_1,...,a_n)$ is 
\begin{equation}\label{eq:loss-sum}
L(\btheta, \ba) = \frac{1}{n}\sum_{i=1}^n\bar{L}(\theta_i,a_i),
\end{equation}
where $\bar{L}$ is some coordinatewise loss function. 
For any decision rule $\delta:\calY^n\to \calA^n$ the risk is defined 
$$
R(\btheta, \delta) = 
\EE_{\btheta} L(\btheta, \delta(\bY)) = 
\frac{1}{n} \EE_{\btheta}  \sum_{i=1}^n \bar{L}(\theta_i,\delta_i(\bY)). 
$$
For example, Stein's normal means problem, to estimate $\theta_1,...,\theta_n$ from independent observations $Y_i\sim \calN(\theta_i,1)$ under squared error is, up to the factor $n^{-1}$ in the loss function, a compound decision problem. 
Robbins noticed that the special structure of a compound decision problem presents an opportunity to design solutions which, in a sense, exhibit optimality at the {\em true} value $\btheta$, sometimes referred to as {\em instance}-optimality. 
He argued as follows. 
Due to the separable form of both the model and the loss, it is reasonable to restrict attention to the class $\calD^S$ of {\em simple} rules, $\delta_i(\bY) = t(Y_i)$ for some $t:\calY\to \calA$. 
Robbins first set out to find the {\it oracle simple rule}, 
\begin{equation}\label{eq:oracle-simple-cd}
\delta_S^* := \argmin_{\delta\in \calD^S}R(\btheta, \delta). 
\end{equation}
Note that, while \eqref{eq:oracle-triv} is often trivial, the oracle simple rule is nontrivial because $\delta$ is constrained. 
For example, in Stein's problem $\delta^*(\bY) \equiv \btheta$ and the corresponding risk is zero, but this is not a simple rule (unless all $\theta_i$ are equal). 
Restricting attention to simple rules, the main insight Robbins provided \citep{robbins1951asymptotically} was relating the original frequentist problem to a {\em postulated} Bayesian problem,   
\begin{equation}\label{eq:model-bayes-simple}
\thetatil_i\stackrel{iid}{\sim} G^*\quad \quad \quad \Ytil_i \lvert \thetatil_i \ \stackrel{ind}{\sim} \ f(\cdot; \thetatil_i), 
\end{equation}
where $G^* = G^*_n$ is the empirical distribution of the (fixed) values $\theta_1,...,\theta_n$, and $f$ is the likelihood from the original model \eqref{eq:model-cd}. 
Indeed, let $\bthetatil = (\thetatil_1,...,\thetatil_n), \bYtil = (\Ytil_1,...,\Ytil_n)$, and take any simple rule $\delta$. 
Then, using the towering property of the  expectation, its Bayes risk under \eqref{eq:model-bayes-simple} is 
\begin{equation}
\label{eq:fundamental-theorem-cd-pf}
\begin{aligned}
\EE L(\bthetatil, \delta(\bYtil)) &= 
\frac{1}{n} \EE \sum_{i=1}^n  \bar{L}(\thetatil_i, \delta_i(\bYtil)) =  
\EE \bar{L}(\thetatil_i, t(\Ytil_i)) = \EE \EE [\bar{L}(\thetatil_i, t(\Ytil_i)) \lvert \thetatil_i] \\
&= 
\frac{1}{n}\sum_{i=1}^n \EE_{\theta_i}\bar{L}(\theta_i, t(Y_i))
=\frac{1}{n} \EE_{\btheta} \sum_{i=1}^n \bar{L}(\theta_i, t(Y_i)) = \EE_{\btheta}L(\btheta, \delta(\bY)), 
\end{aligned}
\end{equation}
i.e., the Bayes risk of a simple rule $\delta$ is precisely the point risk $R(\btheta, \delta)$ in the original compound problem. 
It follows immediately that the oracle simple rule\footnote{It should be noted that the bound \eqref{eq:bound-cd-simple} applies even without independence in \eqref{eq:model-cd}, because the calculation in \eqref{eq:fundamental-theorem-cd-pf} involves only the marginals of the variables $Y_i$. 
In any case, the original definition of a compound decision problem includes the independence assumption in \eqref{eq:model-cd}. 
} in the original problem is the Bayes rule under \eqref{eq:model-bayes-simple},  $\delta^*_S = (t_{G^*}(Y_1),...,t_{G^*}(Y_n))$ where 
$$
t_{G^*}(y) = \argmin_{a\in \calA}\EE[\bar{L}(\thetatil_1, a) |\Ytil_1=y]. 
$$
By definition, the risk of $\delta^*_S$ is the {greatest lower bound} on the risk of any simple rule, 
\begin{equation}
\label{eq:bound-cd-simple}
R(\btheta, \delta^*_S)\leq R(\btheta, \delta)\ \ \ \ \text{for any }\delta\in \calD^S, 
\end{equation}
and, because $\delta^*_S$ has been identified explicitly, we now also have an explicit form of the bound \eqref{eq:bound-cd-simple}. 
Note that for each fixed $\btheta_0\in \Theta^n$, the bound above is indeed attained by the simple rule $\delta_{S}^0 := \argmin_{\delta\in \calD^S}R(\btheta_0, \delta)$, which does not involve the true value $\btheta$ and is therefore a legal decision rule (as opposed to the oracle). 
Unfortunately, $\delta_{S}^0$ does depend on $\btheta_0$, hence, while the bound is attained {\em locally} at each $\btheta_0$, there is no simple rule that attains it {\em uniformly} in $\btheta$. 
From the algorithmic perspective this appears disappointing because, $\btheta$ being unknown, it is unclear what simple rule to choose in practice. 

Nevertheless, Robbins realized that, under some mild conditions, the bound \eqref{eq:bound-cd-simple} can in fact be {\em asymptotically} attained uniformly in $\btheta$, if not by a simple rule. 
His argument proceeds by noting that, while $\delta^*_S$ depends on the true components $\theta_1,...,\theta_n$ of $\btheta$, this dependence is only through their empirical distribution $G^*$. 
This empirical distribution can in turn be {\em estimated} by taking an {\em empirical Bayes} (EB) approach. 
Thus, Robbins proposed to proceed under the working assumption \eqref{eq:model-bayes-simple}, but 
replace $G^*$ with an {unknown} fixed distribution $G$, and {\em estimate} the corresponding Bayes rule $\delta^G_i \equiv t_{G}$. 
In general, such a nonparametric EB estimator can be constructed in a plug-in manner as $\delta^{EB}_i = \widehat{t_G} = t_{\hat{G}}$, where 
$\hat{G}$ is a {\em deconvolution} estimate of $G$, e.g. a nonparametric maximum-likelihood estimator \citep[NPMLE,][]{robbins1950generalization, kiefer1956consistency, chen2017consistency} or an approximation thereof \citep[e.g.][]{efron2016empirical}. 
For some families $f(\cdot; \theta)$ in \eqref{eq:model-cd}, it is actually possible to estimate $t_G$ directly, that is, to bypass estimating $G$ itself \citep{robbins1955empirical, brown2009nonparametric, koenker2014convex, efron2011tweedie}. 
The claim to fame of Robbins's nonparametric EB approach is that $\delta^{EB}$ will, in general, asymptotically attain the risk of the oracle simple rule, 
\begin{equation}
\label{eq:asymp-opt-simple}
R(\btheta; \delta^{EB}) - R(\btheta; \delta^*_S) \ \rightarrow 0\ \ \ \ \ \ \text{as $n\to \infty$}
\end{equation}
for all sufficiently `well-behaved' $\btheta\in \Theta^n$, for example, when $G^*\to G_0$ weakly for some fixed distribution $G_0$  
\citep[][]{robbins1951asymptotically, zhang1997empirical}. 
Hence, $\delta^{EB}$ asymptotically achieves the same risk at the unknown value $\btheta$ as the best simple rule that knows $\btheta$. 
Unless only small $n$ is of interest, this is arguably a much more appealing property than minimaxity or (standard) Bayes. 

As a matter of fact, the situation for compound decision problems is even more favorable than what is described above. 
Competing with the best simple rule seemed reasonable by the separable structure of the problem, but, for example, any EB rule is not simple because the function $\widehat{t_G}(\cdot)$ itself uses all $Y_1,...,Y_n$. 
The bound \eqref{eq:bound-cd-simple} therefore does not apply to $\delta^{EB}$, and one may wonder what improvement is possible by considering non-simple rules. 
A richer class of estimators is the class $\calD^{PI}$ of all {\em permutation invariant} (PI) rules, satisfying 
\begin{equation}\label{eq:cd-pi}
\delta(\bY_\tau) = (\delta(\bY))_\tau\ \ \ \ \ \ \text{for any permutation}\ \tau\in S_n, 
\end{equation}
where $S_n$ is the set of all bijective functions from $[n]:=\{1,...,n\}$ to itself, and where for any $n$-vector $\bx$ we denote $\bx_\tau := (x_{\tau(1)},...,x_{\tau(n)})$. 
Restricting to PI rules can be formally justified by the invariance of the problem itself under permutations, and is the weakest possible condition that can be imposed on a decision rule in a compound decision problem. 
The class of PI rules clearly includes any simple rule, and also any EB rule. 
In \cite{hannan1955asymptotic} the {oracle PI rule}, 
\begin{equation}\label{eq:oracle-PI-cd}
\delta_{PI}^* := \argmin_{\delta\in \calD^{PI}}R(\btheta, \delta), 
\end{equation}
is identified in explicit form for a binary classification compound decision problem. 
More generally, the argument in \cite{greenshtein2009asymptotic} implies that for {\em any} compound decision problem, $\delta_{PI}^*$ is the Bayes rule under the postulated Bayes model
\begin{equation}\label{eq:model-bayes-cd-PI}
\begin{aligned}
\bthetatil 
\sim \text{uniform on all permutations of } \btheta, \qquad \Ytil_i \lvert \thetatil_i 
\stackrel{ind}{\sim} f(\cdot; \thetatil_i).
\end{aligned}
\end{equation}
In analogy to \eqref{eq:bound-cd-simple}, this gives explicitly the greatest lower bound on the risk of any PI rule, 
\begin{equation}\label{eq:bound-cd-PI}
R(\btheta, \delta^*_{PI})\leq R(\btheta, \delta)\ \ \ \ \text{for any }\delta\in \calD^{PI}. 
\end{equation}
%
Because $\calD^S\subset \calD^{PI}$, the bound \eqref{eq:bound-cd-PI} always lies below \eqref{eq:bound-cd-simple}, i.e.~it is a more ambitious lower bound. 
Nevertheless, in compound decision problems it generally holds that, for all sufficiently well-behaved $\btheta\in \Theta^n$, the difference between \eqref{eq:bound-cd-simple} and \eqref{eq:bound-cd-PI} tends to zero asymptotically, 
\begin{equation}\label{eq:asymp-equiv}
R(\btheta, \delta^*_{S}) - R(\btheta, \delta^*_{PI})  \to 0\ \ \ \ \ \ \text{as $n\to \infty$}. 
\end{equation}
The intuition is that, by de Finetti-type results, the model \eqref{eq:model-bayes-cd-PI} should be similar to \eqref{eq:model-bayes-simple} for large $n$: the only difference between them is in the prior, which in the latter samples the components of $\btheta$ with replacement and in the former without replacement. 
Because the difference between sampling with and without replacement should vanish asymptotically in some sense \citep[e.g.][]{diaconis1980finite}, it is reasonable to expect that the Bayes rule under \eqref{eq:model-bayes-cd-PI} would approach the Bayes rule under \eqref{eq:model-bayes-simple}. 
A formal proof of the asymptotic equivalence  \eqref{eq:asymp-equiv} appears in \cite{hannan1955asymptotic} for their binary classification compound decision problem. 
For a compound decision problem with squared loss, \cite{greenshtein2009asymptotic} established under suitable conditions the stronger result $R(\btheta, \delta^*_{S}) - R(\btheta, \delta^*_{PI}) = O(1/n)$, i.e.~that the oracle simple and PI rules are {asymptotically equivalent} under the {sum}, rather than the average, of the coordinatewise losses. 
Combining \eqref{eq:asymp-equiv} with \eqref{eq:asymp-opt-simple} implies that for a compound decision problem, an EB rule asymptotically attains the risk of the oracle PI rule, i.e.
\begin{equation}
\label{eq:asymp-opt-PI}
R(\btheta, \delta^{EB}) - R(\btheta, \delta^*_{PI})  \to  0\ \ \ \ \ \ \text{as $n\to \infty$}, 
\end{equation}
for all sufficiently well-behaved $\btheta\in \Theta^n$. 
This is to say that the ultimate bound \eqref{eq:bound-cd-PI} is in fact asymptotically attainable {\em uniformly} in $\btheta$ by a legal rule. 
Solutions with such a strong optimality property, if only asymptotic, rarely exist in frequentist problems. 
To appreciate how unique compound decision problems are, we emphasize that if $\Theta^n$ is divided into equivalence classes, where two vectors $\btheta$ and $\btheta'$ are equivalent if they are permutations of each other, then each equivalence class generally has a different optimal PI rule; this is because, to use the language of statistical invariance theory \citep[see e.g.][]{berger1985statistical, eaton1989group}, there is no single {\em orbit} with respect to the permutation group (equivalently, it is not {\em transitive}). 
Hence, for any fixed $n$ there does not exist a so-called `best invariant' rule; 
nonetheless, an {\em asymptotically} best  invariant rule in the sense of \eqref{eq:asymp-opt-PI} does exist for compound decision problems.

\subsection{Contributions}\label{sec:contributions}
Compound decision problems are indeed special as far as admitting satisfactory (asymptotic) solutions, but in applications where the model \eqref{eq:model-cd} itself arises, the problems of practical interest are usually {not} compound. 
For example, in many realistic settings the researcher would be interested to estimate only the $\theta_i$'s corresponding to the few largest $Y_i$'s 
instead of estimating all $n$ components $\theta_i$. 
Or, in multiple hypothesis testing, it is common practice to look for procedures that have small type II error while controlling the familywise error rate (FWER) or false discovery rate (FDR) at some preset level. 
Another example frequently encountered in applications is to test a (composite) global null hypothesis such as $H_0: \theta_1=\cdots=\theta_n$. 
Note that none of these inferential tasks can be formulated as a compound decision problem: the first is a {\em selective inference} problem, meaning that the target parameters are determined only after seeing the data, and cannot even be formally written in the conventional decision theoretic framework (under any loss function). 
Global null testing asks to make a single decision, reject or not reject $H_0$, rather than $n$ individual decisions, so this is clearly not a compound decision problem. 
Multiple hypothesis testing does entail making a decision for each $\theta_i$, but FWER and FDR, as well as their type II error counterparts, cannot be written with a loss function of the form \eqref{eq:loss-sum}. 


The thrust of this paper is to show that some of the theory reviewed above, and which was developed specifically for compound decision problems, in fact applies to a much broader class of {\em permutation invariant} (PI) problems that includes all of the examples mentioned in the paragraph above, and many more. 
Informally, a problem is PI according to our  definition if the entire formulation of the problem is symmetric in the $n$ coordinates of $\btheta$. 
Compared to compound decision problems, this relaxes substantially the conditions on both the model \eqref{eq:model-cd} and  the loss \eqref{eq:loss-sum}. 
Importantly, our definition generalizes the conventional definition of a PI problem in that it allows also selective inference problems, as long as the selection rule is itself symmetric with respect to the coordinates of $\btheta$.

For any PI problem in this broad sense, we first identify the oracle PI rule, the minimizer of the risk at $\btheta$ among all procedures that respect the symmetry in the  problem itself. 
This is based on the observation that the considerations yielding the oracle PI rule \eqref{eq:oracle-PI-cd} in the compound decision case, 
continue to hold when instead of \eqref{eq:model-cd} and \eqref{eq:loss-sum} we merely require the model and the loss to be  {\em symmetric} in the $n$ coordinates of $\btheta$. 
Thus, we conclude that for any PI problem per our general definition, the oracle PI rule is again the Bayes rule under the uniform prior on all $n!$ permutations of the true vector $\btheta$. 
At the technical level, this is a relatively straightforward generalization of a standard result under the conventional (non-selective) definition of a PI problem \citep[see, e.g.,][]{greenshtein2009asymptotic}, but it has far-reaching implications. 
To mention just a few examples, if we assume that the model is PI and restrict attention to PI rules, then our bound characterizes the minimum mean squared error (MSE) in estimating the $\theta_i$ corresponding to the largest $M<n$ observations $Y_i$, whereas \eqref{eq:bound-cd-PI} bounds only the MSE for estimating all $n$ parameters. 
Or, in a multiple testing problem, our bound characterizes the maximum attainable power of a testing procedure subject to $\operatorname{FDR}$ control, compared to \eqref{eq:bound-cd-PI} which can only accommodate simpler error metrics such as $\operatorname{mFDR}$ \citep{sun2007oracle}. 
Similarly, our bound indicates the maximum power in testing the global null $H_0: \theta_1=\cdots=\theta_n$, among all tests satisfying the natural condition that permuting $\bY$ does not change the decision to reject or not reject. 

While the bound we obtain on the risk of a PI rule is explicit, 
the calculation is still intractable as it involves a sum over $n!$ terms.
Making this bound useful, then, clearly requires some  approximation, and we present two approaches. 
The first approach is based on a deterministic approximation from \cite{mccullagh2014asymptotic} for the permanent of a doubly stochastic matrix. 
This produces a deterministic alternative to the more straightforward MCMC approach which approximates the oracle risk based on sampling a subset of permutations at random  \citep{greenshtein2019comment}; 
we report our findings in a compound decision example, which shows both strengths and limitations of this approximation. 
The second approach is also asymptotic, but much more `statistical' in that it establishes asymptotic equivalence of the risks of the simple and PI oracle rules in an example that goes {\em beyond} the compound decision setting. 
Specifically, extending a result of \cite{greenshtein2009asymptotic}, we show that, for the suitable definitions of the simple and PI oracle rules, \eqref{eq:asymp-equiv} continues to hold in an example of estimating the parameter of a {\em selected} population, which is not a compound decision problem. 
It should be noted that the result of \cite{greenshtein2009asymptotic} implies immediately that the difference in risks is $O(1/n)$ for a single coordinate $i$ chosen at random independently of the data, but our situation is different because we consider estimating $\theta_{I^*}$, where 
$I^* = \argmax_{i\leq n} Y_i$ is data-dependent. 
Nevertheless, building on their proof, and with some additional assumptions, we show in Theorem \ref{thm:convergence} that the difference in risks still converges to zero. 
To the best of our knowledge, an asymptotic analysis comparing the risks of the simple and PI oracles has not been carried out before beyond the compound decision case. 

The contributions included in this paper concentrate on identifying and studying the greatest lower bound itself in a general PI problem. 
By definition, the bound is attained at each $\btheta$ by the oracle PI rule, which does depend on $\btheta$ in general. 
From the practitioner's viewpoint, showing asymptotic {\em uniform} attainability---i.e., that the bound is asymptotically attainable at every $\btheta$ by a rule that does  not depend on $\btheta$---would of course add much appeal to the results. 
As mentioned earlier, for the special sub-class of compound decision problems, it is known that this is generally achievable using a (nonparametric) EB approach. 
The situation is more complicated when moving to PI problems that are not compound: 
while our result in Theorem \ref{thm:convergence} still suggests using the (usual) EB approach whenever asymptotic equivalence holds between the simple and PI oracles, the property \eqref{eq:asymp-opt-simple} is not  guaranteed to hold in general. 
In fact, even in the specific example we consider here of estimating the parameter of the maximum observation, it is not obvious that \eqref{eq:asymp-opt-simple} remains true under reasonable assumptions. 
This point is discussed in the last section, and deserves further and separate investigation.

\subsection{Organization}
The rest of the paper is organized as follows. 
Section \ref{sec:unified} introduces a generalized decision theoretic framework, that can accommodate almost any statistical task, including selective inference problems. 
Using this extended framework, in Section \ref{sec:pi} we formally define PI problems and PI decision rules. 
In Section \ref{sec:oracle} the attainable lower bound on the risk of any PI rule in a PI problem is derived by finding an explicit expression for the oracle PI rule, the computational aspect is addressed, and an interpretation of the oracle PI rule as a minimizer of a {\em conditional} risk is offered. 
Section \ref{sec:application} specializes the bound to some popular PI problems, and discusses relations to some existing results in the literature. 
In Section \ref{sec:attainability} we prove asymptotic convergence of the risks of the oracle PI and oracle simple rules in a particular PI problem of estimating the mean of a selected population. 
We discuss extensions to models with nuisance parameters in Section \ref{sec:nuisance}, and 
Section \ref{sec:discussion} concludes with a discussion that also includes directions for future research.

\section{A unified decision theoretic framework}\label{sec:unified}

As implied in the previous section, we will want to consider here more general problems than compound decision, relaxing the assumptions on both the model for the data and the statistical task. 
Assume in general that the data $\bY = (Y_1,...,Y_n) \in \calY^n$ comes from some given family of (joint) densities w.r.t.~some dominating measure on $\calY^n$, which is 
 parametrized by the fixed vector $\btheta = (\theta_1,...,\theta_n)\in \bTheta\subseteq \Theta^n$, 
\begin{equation}\label{eq:model}
\bY\sim f(\cdot; \btheta). 
\end{equation}
We assume further that $Y_i\sim f^i_{\theta_i}$, i.e., that the marginal likelihood $Y_i\sim f^i_{\btheta} := \int f(\cdot, \by_{-i}; \btheta)d\by_{-i}$ depends only on $\theta_i$. 
Unlike \eqref{eq:model-cd}, there is no assumption that the $Y_i$'s are independent, or even (at this point) that they all have the same marginal likelihood function. 
Working under this model, we need to make some modifications to the conventional decision theory framework 
in order to accommodate the various statistical problems mentioned in Section \ref{sec:contributions}: 
first of all, in some of the problems there is a {\em validity} condition, for example a confidence interval must cover the true parameter with a preset probability or a hypothesis test must control a Type I error, but in the standard decision theory framework there is conventionally no component allowing us to rule out invalid procedures.\footnote {Note that adjusting the loss function, for example by setting it to infinity in certain cases, will not resolve the issue because the condition is on the risk, not the loss. }
Second, and more importantly, the loss function in the usual framework does not allow the statistician to choose the target of inference adaptively, i.e., as a function of the data. 
Thus, we propose to make two modifications to the standard decision theory framework:
\begin{enumerate}[label=\roman*)]
\item We allow to specify in advance a subset $\calD_0$ of the set of all (measurable) decision rules $\delta: \calY^n\to \calA$, such that only rules in $\calD_0$ can be chosen by the statistician. 
\item We allow the loss function $L:\calY^n\times \bTheta\times \calA \to [0,\infty)$ to take as input the observed vector $\by\in \calY^n$ in addition to $\btheta\in \bTheta$ and $a\in \calA$, i.e.~the loss is $L(\by,\btheta,a)$. 
\end{enumerate}
With appropriate choices for the elements $(\calA, \calD_0, L)$, almost any statistical inference problem under the model \eqref{eq:model} can be formally stated equivalently in the proposed generalized decision theoretic framework, so that a given decision rule in $\calD_0$ is better than another at a particular $\btheta$ in the ordinary sense, if and only if its {\em risk} at $\btheta$, defined henceforth as
$$
R(\btheta,\delta) = \EE_{\btheta}[L(\bY, \btheta, \delta(\bY))], 
$$
is smaller. 
We now demonstrate this for some popular statistical problems that arise under \eqref{eq:model}, including the examples mentioned in Section \ref{sec:contributions}.

\medskip

\begin{enumerate}[wide, labelwidth=!, labelindent=0pt] 
\item {\em Global null testing}. 
Allowing the global null to be in general a {composite} hypothesis, we are interested in testing 
$H_0:\ \btheta\in \boldsymbol{\Theta}_0$ for some $\boldsymbol{\Theta}_0 \subset \bTheta$. 
The action space will be $\calA = \{0,1\}$, where $a=1$ signifies rejection of the null, and the loss will be
$$
L(\btheta,a)=
\begin{cases}
1,&\text{if $\btheta \notin \boldsymbol{\Theta}_0,\ a=0$}\\
0,&\text{otherwise}.
\end{cases}
$$
To enforce validity of a test, we take $\calD_0$ to be the set of all rules $\delta:\calY^n\to \{0,1\}$ s.t. 
$\PP_{\btheta}(\delta(\bY)=1)\leq \alpha$
for all $\btheta\in \boldsymbol{\Theta}_0$, 
%
in other words, only rules $\delta$ that control the Type I error are acceptable. 
Notice that for $\btheta \notin \boldsymbol{\Theta}_0$ the risk is the probability of a Type II error, $\PP_{\btheta}(\delta(\bZ)=0)$, and the problem coincides with the ordinary Neyman--Pearson formulation. 

\item {\em Multiple hypothesis testing}. 
We want to test each of the (composite, in general) hypotheses $H_{0i}:\ \theta_i\in \Theta_0, \ i=1,...,n$, where $\Theta_0\subseteq \Theta$ is prespecified. 
The action space is $\calA = \{0,1\}^n$, where, for $\ba = (a_1,...,a_n)\in \calA$, the coordinate $a_i=1$ if the $i$-th null is rejected and $a_i=0$ otherwise. 
For convenience define also $h_i = \mathbbm{1}(\theta_i\notin \Theta_0), i=1,...,n$. 
The choices of $L$ and $\calD_0$ depend on error criteria under consideration. 
If we define 
$$
\operatorname{fdp}(\btheta, \ba) := \frac{\sum_{i=1}^n (1-h_i) a_i}{\sum_{i=1}^n a_i}, 
$$
with the convention $0/0 = 0$, and take $\calD_0$ to be all rules $\delta:\calY^n\to \calA$ such that $\operatorname{FDR}(\btheta, \delta) := \EE_{\btheta}\operatorname{fdp}(\btheta,\delta(\bY)) \leq \alpha$ for all $\btheta\in \bTheta$, then a procedure is considered valid if it controls the false discovery rate (FDR) at $\alpha$, which is a popular choice in modern literature. 
For type II errors there is no single accepted metric in the literature; one option is to take 
$
L(\btheta, \ba) = \sum_{i=1}^n h_i(1-a_i), 
$
in which case the risk measures type II errors by the expected number of false non-rejections, as considered in, e.g., \cite{heller2021optimal}. 
Other reasonable options include the false non-discovery rate (FNR), the expected ratio of the number of false non-discoveries to the total number of non-discoveries \citep{sun2007oracle}; or the expected ratio of the number of false non-discoveries to the total number of non-nulls \citep{su2017false}. 
To summarize, the choices above for $\calD_0$ and $L$ recover the standard problem of minimizing a type II error rate subject to FDR control. 
It is worth noting that metrics such as $\operatorname{mFNR}:= \EE_{\btheta} [\sum_{i=1}^n h_i(1-\delta_i(\bY))] / \EE_{\btheta} [\sum_{i=1}^n (1-\delta_i(\bY))]$, taking expectations separately in the numerator and denominator, do not immediately fit our framework, because this cannot be expressed as just the expectation of some loss function. 

\item {\em Confidence intervals for selected parameters}. 
Let $s:\mathcal{Y}^n\to \{0,1\}^n$ be a given selection rule, such that $s_i(\bY)=1$ if $\theta_i$ is selected for inference, and $s_i(\bY)=0$ if not. 
The (selective) inference task is to construct a confidence interval (CI) for each of the selected $\theta_i$'s, where a procedure is considered valid if a specified miscoverage error rate is controlled for the constructed CIs. 
For example, \cite{benjamini2005false} proposed to control the false coverage rate (FCR), the expected ratio of noncovering constructed CIs to the total number of constructed CIs. 
Subject to the validity condition, it may be desirable to minimize, e.g., the expected average length of the constructed CIs, or the expected proportion of CIs that cover zero among the constructed CIs \citep{benjamini1998confidence, weinstein2014selective}, or a combination of the two. 
To formalize this in our framework, let $\calA = (\mathcal{P}(\Theta))^n$, with $\mathcal{P}(\Theta)$ the set of all subsets of $\Theta$, and by convention let $a_i = \emptyset$ if a CI is not constructed for $\theta_i$. 
For any decision rule $\delta:\calY^n\to \calA$, define 
$$
\operatorname{fcp}(\by, \btheta, \ba) = \frac{\sum_{i=1}^n s_i(\by) \mathbbm{1}(\theta_i\notin a_i)}{\sum_{i=1}^n s_i(\by)}, 
$$
again with the convention $0/0 = 0$. 
Then we take $\calD_0$ to be all rules $\delta$ for which $\operatorname{FCR}(\btheta, \delta) := \EE_{\btheta}\operatorname{fcp}(\bY, \btheta, \delta(\bY)) \leq \alpha$ for all $\btheta\in \bTheta$. 
To promote procedures that construct short CIs, take 
$$
L(\by, \btheta, \ba) = \frac{\sum_{i=1}^n s_i(\by) \mu(a_i)}{\sum_{i=1}^n s_i(\by)}, 
$$
where $\mu(B)$ denotes the Lebesgue measure of a subset $B\subset \Theta$, assuming $\Theta\subseteq \RR^d$ (for $k=1$ and an interval $B$, this is just the length). 
Hence, the formulation above calls to minimize the expected average length for the constructed CIs subject to FCR control. 

\item {\em Effect-size estimation for selected parameters}. 
To fit the problem into our framework, take $\calA = \Theta^n$, with the interpretation that for $\ba = (a_1,...,a_n) \in \calA$, the coordinate $a_i$ is the estimate of $\theta_i$ in case $\theta_i$ is selected. 
Here $\calD_0$ is unconstrained, i.e.~ it consists of all mappings $\delta: \calY^n\to \calA$. 
To define an appropriate loss, first let $s:\calY^n\to \{0,1\}^n$ be the selection rule, so that $s_i(\bY)=1$ if $\theta_i$ is selected, and $s_i(\bY)=0$ if not. 
If we take, e.g., 
\begin{equation*}
L(\by, \btheta, \ba) = \sum_{i=1}^n s_i(\by)(a_i-\theta_i)^2,
\end{equation*}
 the risk is the mean sum-of-squares error for the selected only. 
Note that, as a special case, this accommodates selection of the indices corresponding to the $M$ largest observations, even when $M$ itself is data-dependent (e.g., if $M$ is the number of rejections made by the level-$\alpha$ Benjamini-Hochberg procedure). 
\end{enumerate}

Many other inferential tasks may be considered, of course, e.g.~multiple sign-classification with familywise \citep{shaffer1980control} or {\em false sign rate} \citep{benjamini1993false, stephens2017false} control; 
estimation of random sums $S_n := \sum_{i=1}^n u(Y_i, \theta_i)$ \citep{zhang2005estimation, greenshtein2008estimating}; or the ranking-based selection problems treated in, e.g., \cite[][]{gu2023invidious}.

\section{Permutation invariant problems}\label{sec:pi}

Working under the framework of Section \ref{sec:unified}, we can now formalize the notions of permutation invariance from the introduction. 
We assume from now on that the set $\calD_0$ of feasible rules is of the form
\begin{equation}\label{eq:d-0}
\calD_0  = \{ \delta: \ \EE_{\btheta}[K(\bY,\btheta,\delta(\bY))] \leq \alpha \ \text{ for all $\btheta\in \bar{\bf \Theta}$}\},
\end{equation}
for some set $\bar{\bf \Theta} \subseteq \bTheta$ that may (but does not have to) depend on the {\em true} value $\btheta$, and for some prespecified function $K(\by,\btheta,a)$ that, similarly to the loss, is defined on $\calY^n\times \bTheta \times \calA$ and returns a nonnegative number. 
We will explain later why this is needed, and why $\bar{\bf \Theta}$ is allowed to depend on $\btheta$; at this point, notice only that in all of the examples of the previous section, $\calD_0$ can indeed be written in this form, with $\bar{\bf \Theta}$ independent of $\btheta$. 

\begin{definition}[Permutation invariant decision problem]\label{def:pi}
Consider a statistical problem in the framework of Section \ref{sec:unified}, with $\calD_0$ of the form \eqref{eq:d-0}. 
We say that the problem is {\it permutation invariant} if: 
\begin{enumerate}[label=(\roman*)]
\item The {\em model} in \eqref{eq:model} is permutation invariant: first, the parameter space $\bTheta$ is closed under permutations, i.e.~for every $\btheta \in \bTheta$ and $\tau\in S_n$ we have $\btheta_{\tau} \in \bTheta$. 
Furthermore, for any $\btheta$ and any $\tau\in S_n$, it holds that ${\bY}_{\tau} \sim f(\cdot; \btheta_{\tau})$. \label{PI-model}
\item The {\em loss function} is permutation invariant: for any $a\in \calA$, any $\tau\in S_n$ and any $\by\in \calY^n$, there exists $a^*\in \calA$ such that $L(\by_{\tau}, \btheta_\tau, a^*) = L(\by, \btheta, a)$ for all $\btheta\in \bTheta$. 
In this case, assuming $L(\by, \btheta, a') = L(\by, \btheta, a)$ for all $\by$ and $\btheta$ implies $a'=a$, the action $a^*\in \calA$ will be unique and we denote it by $\tilde{\tau}(a,\by)$. 
\label{PI-loss}
\item The {\em feasible set} $\calD_0$ is permutation invariant: the function $K$ in \eqref{eq:d-0} has the same property as $L$ in the previous item, and the set $\bar{\bf \Theta}$ is PI, meaning that 
$\btheta\in \bar{\bf \Theta}$ implies $\btheta_{\tau}\in \bar{\bf \Theta}$ for any $\tau\in S_n$, any $\btheta\in \bTheta$. 
\end{enumerate}
\end{definition}


This definition generalizes the usual definition of a PI problem \citep[e.g.,][Ch. 6.2]{berger1985statistical}, the main difference being that Definition \ref{def:pi} applies to loss functions in the generalized sense, i.e., the loss may depend also on $\bY$. 
Remembering that $\btheta$ always indexes the distribution of $\bY$ (not  $\bY_\tau$), condition \ref{PI-model} says that the distribution of $\bY_\tau$ under $\btheta$ is the same as the distribution of $\bY$ under $\btheta_{\tau}$. 

Note that a compound decision problem is only a very special case of a PI problem as defined above. 
First, for any family of univariate distributions $f(\cdot; \theta), \theta\in \Theta$, the model \eqref{eq:model-cd} is PI, for example if the independent observations are $Y_i\sim \calN(\theta_i, 1)$ (but not if $Y_i\sim \calN(\theta_i, \sigma^2_i)$ for known $\sigma^2_i$, unless they are all equal). 
However, independence is not necessary: for example, if $Y_i = \theta_i + \epsilon_i$, with $\epsilon_i$ zero-mean exchangeable random variables (in particular, if the errors vector is an equi-correlated gaussian), or if $\bY$ is multinomial with parameter $\btheta\in \Delta_{n-1}$, where $\Delta_{n-1} = \{\btheta\in [0,1]^n: \sum \theta_i=1\}$, then the $Y_i$'s are not independent but the model is still PI. 
Second, for any $\bar{L}$ the loss in \eqref{eq:loss-sum} is clearly PI, but the condition (ii) in our definition is much weaker, requiring only symmetry of the loss function w.r.t.~the coordinates $i=1,...,n$ of the problem. 


By the general invariance principle, permutation invariant problems naturally call to restrict attention to permutation invariant {\em rules}\footnote{It is more accurate to refer to rules satisfying \eqref{eq:pi-rule}, which in the compound decision case reduces to \eqref{eq:cd-pi}, as permutation {\em equivariant}, but we keep ``invariant" for terminological consistency.}, namely, to procedures $\delta$ satisfying that 
%
for every $\bY\in \calY^n$ , 
\begin{equation}\label{eq:pi-rule}
\delta(\bY_{\tau}) = \tilde{\tau}(\delta(\bY), \bY)\ \ \ \ \ \ \ \text{for all $\tau\in S_n$}, 
\end{equation}
and recall the definition of $\tilde{\tau}$ in Definition \ref{def:pi}. 
In the special case of a compound decision problem $\tilde{\tau}(\ba,\by) = \ba_{\tau}$, and \eqref{eq:pi-rule} becomes \eqref{eq:cd-pi};  we need the more general form \eqref{eq:pi-rule}, involving $\tilde{\tau}$, for the principle to apply beyond the compound decision setting, i.e.~to generalize the condition to every PI problem per Definition \ref{def:pi}. 
For any PI problem, we let $\calD^{PI} $ be the collection of all PI rules, and we let $\calD^{PI}_0 = \calD^{PI} \cap \calD_0$ be the collection of all {\em feasible} PI rules. 

Because adding the PI requirement on $\delta$ is so natural when the problem itself is PI, we consider $\calD^{PI}_0$ to be really the largest class of decision rules one could possibly consider. 
The following simple property of PI rules will be leveraged in the next section. 

\begin{lemma}\label{lem:orbits}
Consider a decision problem in the framework of Section \ref{sec:unified}. 
If the problem is PI,  then for any PI rule $\delta$ and for any $\btheta$, 
$$
R(\btheta, \delta) = R(\btheta_\tau, \delta)\ \ \ \ \ \ \text{for all $\tau\in S_n$}.
$$
\end{lemma}

\begin{proof}
Recall that the risk of any rule $\delta$ is defined as 
$$
R(\btheta, \delta) \!= \! \EE_{\btheta}[L(\bY, \btheta, \delta(\bY))] \!= \!\!\int \!\!L(\by, \btheta, \delta(\by)) f(\by;\btheta) d \by. 
$$
Then, we have
$$
\begin{aligned}
R(\btheta_{\tau}, \delta) = \EE_{\btheta_{\tau}} L(\bY, \btheta_{\tau}, \delta(\bY)) = \EE_{\btheta} L(\bY_{\tau}, \btheta_{\tau}, \delta(\bY_{\tau})) 
&= \EE_{\btheta} L(\bY_{\tau}, \btheta_{\tau}, \tilde{\tau}(\delta(\bY))) \\
& = \EE_{\btheta} L(\bY, \btheta, \delta(\bY))= R(\btheta, \delta), 
\end{aligned}
$$
where the second equality holds because the model is PI, the third equality is because $\delta$ is PI, and the fourth equality is because the loss is PI. 
\end{proof}

Lemma \ref{lem:orbits} says that the risk of any PI rule is constant on any {\em orbit}, 
$
\mathcal{O}(\btheta) = \{\btheta_\tau:\ \tau\in S_n\} \subseteq \bTheta. 
$
For PI problems per the conventional definition in the literature (in particular, for a compound decision problem), 
Lemma \ref{lem:orbits} recovers a basic result in the theory of invariance expressing the best invariant rule in terms of the appropriate Haar measure \cite[][Ch. 6.2]{berger1985statistical}.

\section{An explicit form of the oracle PI rule}\label{sec:oracle}
Consider any PI problem per Definition \ref{def:pi}. 
Our aim in this section is to identify explicitly the {\em oracle PI rule}, 
\begin{equation}\label{eq:oracle-PI}
\delta_{PI}^* := \argmin_{\delta\in \calD_0^{PI}} R(\btheta, \delta). 
\end{equation}
By the definition of $\delta_{PI}^*$, we have 
\begin{equation}\label{eq:bound-PI}
R(\btheta, \delta^*_{PI})\leq R(\btheta, \delta)\ \ \ \ \text{for any }\delta\in \calD^{PI}_0, 
\end{equation}
i.e., the left hand side is the greatest lower bound on the risk of any feasible PI rule. 
Hence, if we find an explicit form of $\delta_{PI}^*$ we also have an explicit form of the bound \eqref{eq:bound-PI}. 
This is a trivial observation, but it is useful because, whereas $\delta_{PI}^*$ depends on $\btheta$ and is therefore not a legal decision rule, $R(\btheta, \delta^*_{PI})$ is a perfectly legal lower bound on the risk. 

In the special case of a compound decision problem, \eqref{eq:oracle-PI} reduces to \eqref{eq:oracle-PI-cd}. 
As reviewed in Section \ref{sec:intro}, in this case it is a known result that $\delta_{PI}^*$ is given explicitly by the Bayes rule under a (postulated) prior that assigns equal mass to each of the $n!$ permutations of the fixed vector $\btheta$. 
We now generalize this result, proving that in {\em any PI problem}, the oracle PI rule $\delta_{PI}^*$ is again the minimizer of the Bayes risk under that same prior. 

Thus, assume that $\bY$ follows a model of the form \eqref{eq:model}, and that this model is PI. 
Now define a random pair $(\bthetatil,\bYtil)$ distributed according to
\begin{equation}\label{eq:bayes-PI}
\tau \sim \text{uniform on }S_n, \quad \quad (\bthetatil, \bYtil)|\tau \sim (\btheta_\tau, \bY_\tau). 
\end{equation}
Note that this distribution depends on the fixed $\btheta$ up to a permutation. 
From now on we write $\EE$ for expectation w.r.t.~$(\bthetatil,\bYtil)$ under the (Bayes) model \eqref{eq:bayes-PI}, remembering that it still depends on $\btheta$ up to permutation, and we write  $\EE_{\btheta}$ for expectation w.r.t.~$\bY$ under the original (frequentist) model \eqref{eq:model}. 
Now suppose there exists a solution $\deltabayes$ to the optimization problem given by 
\begin{equation}\label{eq:optimization}
\begin{aligned}
\displaystyle{\minimize_{\delta}} \quad &\EE L(\bYtil, \bthetatil, \delta(\bYtil))\\
\textup{s.t.} \quad &\EE K(\bYtilprime, \bthetatilprime, \delta(\bYtilprime)) \leq \alpha\ \ \forall \btheta'\in \bar{\bf \Theta}, 
\end{aligned}
\end{equation}
where $(\bthetatilprime,\bYtilprime)$ has the distribution of $(\bthetatil,\bYtil)$ when $\btheta$ is replaced by $\btheta'$. 
In this case, the following result says that $\deltabayes$ is the oracle PI rule in the original problem. 

\begin{theorem}\label{thm:main}
In the framework of Section \ref{sec:unified}, consider a PI problem with loss $L(\by,\btheta,a)$ and an associated set 
$\calD_0$ 
of feasible rules given by \eqref{eq:d-0}. 
Then for any fixed $\btheta$ we have $\delta^*_{PI} = \deltabayes$, that is, 
 $R(\btheta, \deltabayes) = \min_{\delta\in \calD^{PI}_0}R(\btheta, \delta)$. 
\end{theorem}

\begin{proof}
We need to show that $R(\btheta, \deltabayes) = R(\btheta, \delta^*_{PI})$. 
By definition, $\delta^*_{PI}$ is the solution to 
\begin{equation}\label{eq:delta-star-def}
\begin{aligned}
\displaystyle{\minimize_{\delta\in \calD^{PI}_0}} \quad &R(\btheta, \delta)\\
\textup{s.t.} \quad &\EE_{\btheta'}[K(\bY, \btheta', \delta(\bY))]\leq \alpha\ \ \forall \btheta' \in \bar{\boldsymbol{\Theta}}. 
\end{aligned}
\end{equation}
By Lemma \ref{lem:orbits}, for any PI rule $\delta$, 
$$
\begin{aligned}
R(\btheta, \delta) = \frac{1}{n!}\sum_{\tau\in S_n} R(\btheta_\tau, \delta) = 
\EE R(\bthetatil, \delta) 
= \EE  \EE [L(\bYtil,\bthetatil,\delta(\bYtil)) | \bthetatil] = \EE L(\bYtil,\bthetatil,\delta(\bYtil)). 
\end{aligned}
$$
Because the function $K$ has the same properties as $L$, and because $\bar{\boldsymbol{\Theta}}$ is closed under permutations by assumption, repeating the argument above yields that for any PI rule $\delta$ and any $\btheta'\in \bar{\boldsymbol{\Theta}}$, 
$$
\begin{aligned}
\EE_{\btheta'}K(\bY, \btheta', \delta(\bY)) = 
\EE K(\bYtilprime,\bthetatilprime,\delta(\bYtilprime)).
\end{aligned}
$$
Therefore, \eqref{eq:optimization} and \eqref{eq:delta-star-def} are equivalent, except that in \eqref{eq:optimization} there is no restriction that $\delta$ is PI. 
It is easy to verify, however, that from symmetry of the prior, the likelihood, the objective and the constraint in \eqref{eq:optimization} with respect to the coordinates $1,...,n$, it follows that the {\em solution} to \eqref{eq:optimization} will in fact be PI. 
The claim follows. 
\end{proof}

As stated, Theorem \ref{thm:main} is general enough to accommodate all of the different inference problems mentioned in Section \ref{sec:unified}, and characterizes the oracle PI rule as a constrained minimizer of a Bayes risk under the postulated model \eqref{eq:bayes-PI}, assuming the latter exists. 
In particular, when $\calD_0$ is unrestricted, i.e.~if $\calD_0^{PI} = \calD^{PI}$, define 
$$
\rho(\by,a) = \EE [L(\by, \bthetatil, a)\lvert \bYtil=\by]
$$
to be the {\em selective} posterior expected loss of action $a\in \calA$ under \eqref{eq:bayes-PI}; 
this is different from the posterior expected loss in the usual decision theoretic framework, in that $L$ is a loss function in the broader sense of Section \ref{sec:unified}, i.e.~it may depend on $\by$. 
Now let 
\begin{equation}\label{eq:oracle-PI-unrestricted}
\deltabayes(\bY) = \argmin_{a\in \calA} \rho(\bY,a) 
\end{equation}
be the minimizer of the selective posterior expected loss, assuming it exists. 
Then by standard results from Bayesian theory, the rule \eqref{eq:oracle-PI-unrestricted} minimizes the objective in \eqref{eq:optimization}, assuming a rule with finite risk exists in the original problem; see \cite[][Theorem 4.1.1]{lehmann2006theory}. 
If we further assume that $L(\by, \btheta, a)$ is strictly convex in $a$, that $\EE L(\bYtil, \bthetatil, \deltabayes(\bYtil))$ is finite, and that $f(\cdot;\btheta)$ is absolutely continuous w.r.t.~the marginal distribution of $\bYtil$ for every $\btheta$, then $\deltabayes$ above is unique \citep[][Corollary 4.1.4]{lehmann2006theory}. 
Under these conditions, Theorem \ref{thm:main} implies immediately 

\begin{corollary}\label{cor:unrestricted}
In the setting of Theorem \ref{thm:main}, suppose $\calD_0$ is unrestricted. 
Then, assuming the conditions just stated, the oracle PI rule $\delta^*_{PI}(\bY)$ is given by \eqref{eq:oracle-PI-unrestricted}. 
\end{corollary}

As we demonstrate in the next section, this result has useful implications for selective inference PI problems. 
Recall from Section \ref{sec:unified} that our framework accommodates selective inference problems by using a loss function that depends on $\by$. 
In that case, to obtain the oracle PI rule, Corollary \ref{cor:unrestricted} says that we need to update the distribution of $\bthetatil$ given $\bYtil$ under \eqref{eq:bayes-PI} while {\em ignoring} selection, and minimize the posterior expected loss on the selected parameters only. 
Essentially, this is another manifestation of the indifference to selection in a Bayesian framework; but allowing the loss function itself to depend on the data, as we do here, takes fuller advantage of the fact that the Bayes rule anyway conditions on the data. 

\subsection{Computation}\label{subsec:comp}

While an explicit form is now available for the lower bound $R(\btheta, \delta^*_{PI})$, computation is still a concern because the exact expression  involves a sum over $n!$ terms. 
From the computational perspective there is no essential difference between calculating $R(\btheta, \delta^*_{PI})$ or the oracle itself, because once we have an expression for $\delta^*_{PI}(\by)$ the risk can be approximated by averaging the loss over $\bY$ either by simulating repeated draws of $\bY\sim f(\cdot; \btheta)$ or by numerical integration. 
We focus here on the simplest setup, assuming the independent model \eqref{eq:model-cd} and an unrestricted $\calD_0$. 
For example, consider estimating $\btheta$ from independent data $Y_i\sim \text{Pois}(\theta_i), i\leq n$, under sum-of-squares loss divided by $n$, so the oracle rule is 
$$
\delta^*_{i, PI}(\by) = \frac{\sum_{\tau\in S_n}\theta_{\tau(i)}\prod_{i=j}^n f_{\tau(j)}(y_j)}{\sum_{\tau\in S_n} \prod_{j=1}^n f_{\tau(j)}(y_j)},
$$
where $f_j(y_i):= f(y_i; \theta_j) = e^{-\theta_j} {\theta_j}^{y_i}/y_i!$. 
One approach for approximating the oracle risk is by simulation. 
Let $\delta^*_{sim}(\by)$ denote the estimator obtained by replacing $S_n$ in the above expression for $\delta^*_{i, PI}(\by)$ with a subset $S$ of $N$ randomly drawn permutations. 
This randomized estimator is PI in the broader sense that its risk (expectation taken over $S$ and $\bY$) does not depend on the ordering of $\btheta$, hence its risk upper bounds $R(\btheta, \delta^*_{PI})$; an MCMC estimate of $R(\btheta, \delta^*_{sim})$ will thus be an upwardly biased estimate of $R(\btheta, \delta^*_{PI})$. 
This seems to be the approach taken by \cite[][]{greenshtein2019comment} when producing the simulation results reported in their Figure 1. 
We propose an alternative approach, giving a {\em deterministic} but asymptotic approximation to the PI oracle; 
the following is an adaptation of ideas that I learned from D.~Xiang (private communication). 
The expression for $\delta^*_{PI, 1}(\by)$ in the display above can be written equivalently 
$$
\delta^*_{i, PI}(\by) = 
 \frac{ \sum_{j=1}^n \theta_j f_j(y_i) \sum_{\tau\in S_n:\tau(i)=j} \prod_{k \neq i} f_{\tau(k)}(y_k)}{\sum_{\tau\in S_n} \prod_{k=1}^n f_{\tau(k)}(y_k)}. 
$$ 
Here the denominator is exactly the permanent of the $n\times n$ matrix $A$ with entries $A_{ij} = f_j(y_i)$, and in the numerator the $j$th summand involves the permanent of a $(n-1)\times (n-1)$ matrix, $A^{(i,j)}$, obtained from $A$ by deleting the $i$th row and the $j$th column. 
Now, the matrix $A$ has positive entries, so it can be represented as $A = RBC$, where $B$ is doubly stochastic (positive entries, with unit row and column sums) and $R, C$ are diagonal with positive entries. 
Using the deterministic asymptotic
 approximation in \cite{mccullagh2014asymptotic} for the permanent of a doubly stochastic matrix, 
\begin{equation}\label{eq:perm-approx}
\log \operatorname{per}(n B) \sim \log (n!)-\frac{1}{2} \log \operatorname{det}\big(I_n+t^2 J-t^2 B^{\top} B \big), 
\end{equation}
with $J=n^{-1} 11^\top$ and $t^2 = n/(n-1)$, we can approximate $\operatorname{per}(A) \approx \operatorname{per}(B)\prod_{i=1}^n R_{ii}C_{ii}$. 
Applying this approximation also to the permanents of $A^{(i,j)},\ j=1,...,n$, yields an approximation to $\delta^*_{i, PI}(\by)$. 
All of this assumes the determinant term in \eqref{eq:perm-approx} is positive, otherwise the log is undefined. 

We examined the accuracy of our proposed approximation in a simulation with Poisson likelihood. 
We generated $\theta_i\sim \text{Gamma}(\text{shape}=3,\text{rate}=1)$, indep.~for $i=1,...,n$, and held these values fixed throughout the experiment. 
In each run we draw $\tau$ uniformly on all permutations of $[n]$, then draw $Y_i\sim \text{Pois}(\theta_{\tau(i)})$, indep. 
For any PI estimator the overall risk is estimated by the average of $(\delta_{1,PI}(\bY)-\theta_{\tau(1)})^2$ over $m$ simulation runs. 
For $n=10$, where an exact calculation of the PI oracle estimator is still possible, the absolute percent difference between $\delta_{1,PI}(\bY)$ and the approximation, in $m=100$ simulation rounds, did not exceed 6\%, with mean 1.1\% and median 0.8\%, and in all 100 runs the approximation evaluated successfully, i.e.~the determinant term in \eqref{eq:perm-approx} was positive. 
The estimated risk was 1.69 for the exact oracle and 1.71 for the approximate version. 
These findings demonstrate that the approximation, when applicable, can be remarkably good already for small $n$. 
Figure \ref{fig:ratio} shows estimates of the ratio $R(\btheta, \delta^*_{PI})/R(\btheta, \delta^*_{S})$ for different values of $n$, each point in the plot obtained on $m=1000$ simulation runs (the oddity at $n=40$ seems to be due to random fluctuation). 
The numerator in the ratio is estimated using the proposed deterministic approximation, the denominator uses the exact form of the simple oracle rule. 
In this simulation the approximation evaluated successfully on all 1000 runs for $n=10, 20, ..., 140$; when we tried $n=150$ , the approximation failed on all 1000 runs, i.e., the determinant in \eqref{eq:perm-approx} was negative. 
We encountered issues with the approximation failing also when trying the normal example from \cite{greenshtein2019comment}, $\theta_i\sim \calN(0,9)$, $Y_i\sim \calN(\theta_i, 1)$, although this time the approximation failed more frequently for small, rather than large, values of $n$. 
It is worth mentioning that for this normal example with $n=100$, the computation evaluated successfully on 94\% of $m=1000$ simulation rounds, and the mean of the risk ratio on the successful runs was 0.98, significantly higher than the number $\approx 0.8$ reported  in \cite{greenshtein2019comment}. 
The technical issue with the computation failing of course limits the usefulness of the proposed approach for approximating our lower bound \eqref{eq:bound-PI}; but the encouraging results we obtained on the Poisson example motivate further exploration of the applicability of Mccullagh's formula, perhaps by considering higher order terms.

\begin{figure}
  \centering
\includegraphics[scale=.5]{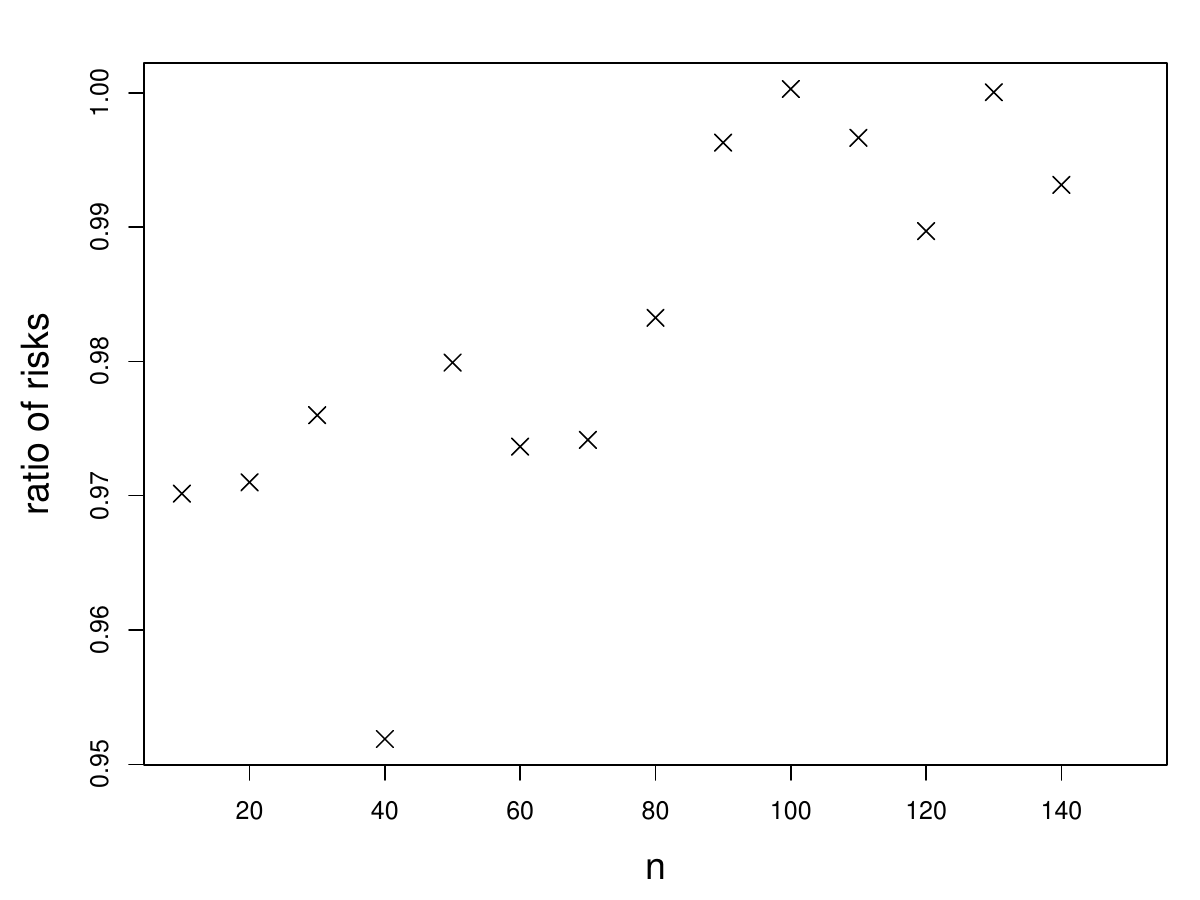} 
\caption{
Estimated ratio of the PI oracle risk to the simple oracle risk, Poisson example. 
}
\label{fig:ratio}
\end{figure}

%

\subsection{Minimizing a conditional risk}\label{subsec:conditional}
The oracle PI rule characterized in Theorem \ref{thm:main} minimizes the point risk for each $\btheta$, but if considered in a slightly modified problem, it actually has stronger, conditional optimality guarantees. 
Consider any PI problem where $\calD_0$ is unrestricted, i.e.~it consists of all measurable mappings $\delta:\calY^n \to \calA$. 
In this case, as stated in Corollary \ref{cor:unrestricted}, the oracle PI rule is the usual Bayes rule under \eqref{eq:bayes-PI}, that is, in terms of the hypothetical variables $(\bYtil, \bthetatil)$ it minimizes not only the expectation of the loss $\EE [L(\bYtil, \bthetatil, \delta(\bYtil))]$, but also the expectation of the loss {\em conditionally} on $\bYtil$. 
In the original problem, where $\btheta$ is fixed, it is meaningless to condition on $\bY$ because this would remove all randomness from the problem, and the guarantees in Theorem \ref{thm:main} (and Corollary \ref{cor:unrestricted}) are indeed unconditional. 
Nevertheless, the oracle PI rule can be shown to minimize a {\em conditional} risk in the fixed-$\btheta$ setting, if we slightly modify the formulation of the problem. 
Thus, let $\bY_{(1:n)}:= (Y_{(1)}\leq \cdots \leq Y_{(n)})$ be the order statistic of $\bY$, and let $\btheta_{(1:n)} = (\theta_{(1)},...,\theta_{(n)})$ be the $\theta_i$ ordered accordingly; 
formally, if $\pi$ is the permutation that sorts $Y_{\pi(1)}\leq \cdots \leq Y_{\pi(n)}$, then $\theta_{(i)} := \theta_{\pi(i)}$. 
Note that, while $\theta_i$ are strictly fixed, $\theta_{(i)}$ are random (because $\pi$ is random). 
Extending the ideas proposed by \citet[][Section 6]{genovese2002operating} in a multiple testing context, restricting attention to PI rules in a PI problem is equivalent to operating on the {\em unlabeled} data, i.e., the statistician observes only the order statistics $Y_{(i)}$, and makes decisions regarding the corresponding (random) parameters $\theta_{(i)}$. 
We emphasize that the unlabeled viewpoint is justified precisely by the permutation invariance of the problem, since in a PI problem labels are immaterial. 
Now let $\delta$ be any decision rule in the unlabeled problem (not PI, because the unlabeled problem is not PI). 
Then it is easy to verify that, for any fixed $\bt = (t_1\leq \cdots \leq t_n)$, we have
$$
\begin{aligned}
\EE_{\btheta} [L(\bY_{(1:n)}, \btheta_{(1:n)}, \delta(\bY_{(1:n)})) \lvert \bY_{(1:n)}  = \bt] 
= 
\EE [L(\bYtil, \bthetatil, \delta(\bYtil)) \lvert \bYtil = \bt], 
\end{aligned}
$$
where $(\bYtil, \bthetatil)$ are the variables in the hypothetical Bayes model \eqref{eq:bayes-PI}. 
But the minimizer of the right hand side is the oracle PI rule in the original problem, 
 $\delta^*_{PI}(\bt) = \argmin_{a\in \calA}  \EE [L(\bYtil, \bthetatil, a)\lvert \bYtil=\bt]$. 
Hence, $\delta^*_{PI}(\bt)$ also minimizes the left hand side, i.e., $\delta^*_{PI}(\bY_{(1:n)})$ minimizes the risk in the unlabeled problem {\em conditionally} on $\bY_{(1:n)}$. 
Genovese and Wasserman \cite[][GW02 henceforth]{genovese2002operating} considered a very specific PI multiple testing problem where the same simple null is tested against the same simple alternative in each coordinate, under the loss $FNP + \lambda \cdot FDP$, and the $Y_i$ are p-values. 
In this case the (conditionally) optimal procedure rejects the $R = R(\bY_{(1:n)})$ hypotheses with smallest p-values, and GW02 also derive formulas to specify $R = R(\bY_{(1:n)})$, so their solution is even more explicit than our characterization of $\delta^*_{PI}$. 
On the other hand, the characterization of $\delta^*_{PI}$ in Corollary \ref{cor:unrestricted} applies in a much broader class of problems. 
In particular, for the multiple testing problem, if we move from a simple alternative to a composite alternative, the calculations of GW02 are inapplicable, whereas our Corollary \ref{cor:unrestricted} implies, by similar arguments to those in the second example of Section \ref{sec:application}, that the conditionally optimal procedure (for the same loss considered in GW02) rejects a subset of hypotheses for which the local-FDR statistic in Equation \eqref{eq:lfdr} is sufficiently small. 
Importantly, this is generally different from the procedure of Genovese and Wasserman which thresholds the p-value: the two procedures agree in their setup where both null and alternative are simple, but not in this more general setup.

\section{Application in some example problems}\label{sec:application}

In this section we specialize the general results to the examples of Section \ref{sec:unified}, when the model (and, hence, the entire problem) is PI. 
To remind, in those examples we reformulated some popular problems in simultaneous inference within the proposed unified framework. 
For instance, the version of the multiple testing problem presented in Section \ref{sec:unified}, corresponds to $K(\by,\btheta,\ba) = \operatorname{fdp}(\btheta, \ba)$ and $\bar{\bf \Theta} = \bTheta$ in \eqref{eq:d-0}, so $\calD_0$ is restricted to all procedures $\delta$ satisfying $\operatorname{FDR}(\btheta, \delta)\leq \alpha$ for {\em all} $\btheta\in \bTheta$; this is the standard requirement for Type I error control in the FDR problem, known as {\em strong control}. 
Before applying the results from the previous section to these examples, we propose a subtle but important modification, which applies when $\calD_0$ is restricted: 
\begin{center}
whenever $\bar{\bTheta} = \bTheta$ in \eqref{eq:d-0}, replace it by $\bar{\bTheta} = S_n(\btheta)$, 
\end{center}
the set of $n!$ (possibly nondistinct) permutations of the true $\btheta$. 
With this modification, the constraint 
$$
\EE K(\bYtilprime, \bthetatilprime, \delta(\bYtilprime)) \leq \alpha\ \ \forall \btheta'\in \bar{\bf \Theta}
$$ 
in \eqref{eq:optimization} becomes just $\EE K(\bYtil, \bthetatil, \delta(\bYtil)) \leq \alpha$, because the distribution of $(\bYtil, \bthetatil)$ is invariant to permutations of $\btheta$. 
Theorem \ref{thm:main} then says that the optimal PI rule is the minimizer of 
\begin{equation}\label{eq:optimization-mod}
\begin{aligned}
\displaystyle{\minimize_{\delta}} \quad &\EE L(\bYtil, \bthetatil, \delta(\bYtil))\\
\textup{s.t.} \quad &\EE K(\bYtil, \bthetatil, \delta(\bYtil)) \leq \alpha, 
\end{aligned}
\end{equation}
which is simpler to solve compared to \eqref{eq:optimization}, because the same pair $(\bYtil, \bthetatil)$ appears in both the objective and the constraint. 
The resulting oracle PI rule does solve a slightly different problem when our modification is employed. 
For example, in the multiple testing problem, for every $\btheta$ this will give the optimal PI testing procedure that controls the FDR only at the {\em true} value $\btheta$ (or permutations thereof), rather than the optimal PI procedure that controls FDR in the strong sense, as originally prescribed. 
Still, there should not really be an objection to the proposed modification, for several reasons: 
(i) the purpose of requiring validity of a procedure (for example type I error control in multiple testing) for all $\btheta\in \bTheta$, is anyway to ensure that it is valid at the unknown, true instance $\btheta$; 
(ii) the oracle PI rule is anyway optimal only at the {\em true} value $\btheta$ (or permutations thereof), even without our modification, since the objective in \eqref{eq:optimization} anyhow depends on the true $\btheta$; 
(iii) our modification only makes for a {\em stronger} oracle, in other words the corresponding  bound \eqref{eq:bound-PI} can only become smaller when taking $\bar{\bf \Theta} = S_n(\btheta)$ instead of $\bar{\bf \Theta} = \bTheta$. 

We return to the examples of Section \ref{sec:unified}, assuming for the rest of the section that the model \eqref{eq:model} itself is PI. 
For each of these examples, we characterize explicitly the oracle PI rule by applying Theorem \ref{thm:main}, incorporating the modification above when applicable. 

\begin{enumerate}[wide, labelwidth=!, labelindent=0pt] 
\item {\em Global null testing}. 
For the general formulation in Section \ref{sec:unified} to entail a PI problem, we need the null set $\boldsymbol{\Theta}_0$ to be PI, i.e., closed under permutations. 
For simplicity, we consider the most basic version of a point null hypothesis, $H_0: \theta_i=0 \ \text{for all }i=1,...,n$, i.e.~$\boldsymbol{\Theta}_0 = \{\boldsymbol{0}\}\subseteq \RR^n$, but the arguments can also be applied, for example, to testing $H_0: \theta_1=\theta_2=\cdots=\theta_n$ because this also corresponds to a PI set $\boldsymbol{\Theta}_0$. 
Theorem \ref{thm:main} says that $\delta^*_{PI}$ is the solution to 
\begin{equation*}
\begin{aligned}
\displaystyle{\minimize_{\delta}} \quad &\PP_{\btheta}(\delta(\bYtil)=0)\\
\textup{s.t.} \quad &\PP_{\btheta=\boldsymbol{0}}(\delta(\bYtil)=1)\leq \alpha. 
\end{aligned}
\end{equation*}
Recall that $\bYtil$ is a random vector obtained by mixing $\bY\sim f(\cdot; \btheta)$ according to the uniform distribution on all permutations of $\btheta$, so that the problem posed above in terms of $\bYtil$, is a {\em simple versus simple} testing problem in the usual Neyman--Pearson paradigm. 
We conclude that $\delta^*_{PI}$ is the corresponding likelihood ratio test, 
\begin{equation}\label{eq:oracle-global-testing}
\delta^*_{PI}(\bY) \!= \!
\begin{cases}
1, &\text{if }\tilde{\Lambda}(\bY)\leq c\\
0, &\text{otherwise}
\end{cases}, \qquad \tilde{\Lambda}(\tilde{\by}) = \frac{f_{\bzero}(\tilde{\by})}{\frac{1}{n!}\sum_{\tau\in S_n}f(\tilde{\by}; \btheta_\tau)},
\end{equation}
where $f_{\bzero}$ is the density of $\bYtil$ (also, of $\bY$) under $\btheta=\bzero$, and 
where $c$ is the largest constant s.t.~$\PP_{\btheta=\bzero}(\tilde{\Lambda}(\bY)\leq c)\leq \alpha$. 

\item {\em Multiple hypothesis testing}. 
The setup is identical to that presented in Section \ref{sec:unified}, except that, per our modification, $\calD_0$ is now the set of all $\delta:\calY^n\to \{0,1\}^n$ s.t.~$\operatorname{FDR}(\btheta, \delta) \leq \alpha$, i.e.~only at the true value of the parameter. 
Applying \eqref{eq:optimization-mod}, the oracle PI rule solves
\begin{equation}
\begin{aligned}
\displaystyle{\minimize_{\delta}} \quad &\EE L(\bthetatil, \delta(\bYtil)) \\
\textup{s.t.} \quad &\EE \operatorname{fdp}(\bthetatil, \delta(\bYtil))\leq \alpha, 
\end{aligned}
\end{equation}
where, consistent with Section \ref{sec:unified}, 
$$
L(\bthetatil, \delta(\bYtil)) = \sum_{i=1}^n \htil_i(1-\delta_i(\bYtil)),\qquad \quad
\operatorname{fdp}(\bthetatil, \delta(\bYtil)) = {\sum_{i=1}^n (1-\htil_i) \delta_i(\bYtil) \big/ \sum_{i=1}^n \delta_i(\bYtil)},
$$
for $\htil_i = \mathbbm{1}(\thetatil_i\notin \Theta_0)$. 
Thus, we want to minimize the expected number of false non-rejections subject to FDR control, but now the expectation is taken also w.r.t.~the random variables $\thetatil_i$. 
Specifically, recalling that $\bthetatil$ is a random permutation of $\btheta$, we have that $\tilde{\boldsymbol{h}} = (\htil_1,...,\htil_n)$ is a random permutation of $\boldsymbol{h} = (h_1,...,h_n)$. 
In particular, 
\begin{equation}\label{eq:htil}
\htil_i\sim \textnormal{Bernoulli}(\pi_1), 
\end{equation}
with $\pi_1=\pi_1(\btheta) = \frac{1}{n}\sum_{i=1}^n h_i$, giving rise to a version of the two-groups model \citep{efron2001empirical}, although the $\htil_i$'s above are only exchangeable random variables, as opposed to the standard two-groups model or even its extension in \cite[][HR21 henceforth]{heller2021optimal}, where the $\htil_i$'s are assumed independent. 
This framing allows us to leverage some existing results on optimal FDR controlling procedures under a two-groups model. 
Specifically, HR21 considered a `general' two-groups model, assuming \eqref{eq:htil} with independence, for a fixed and known $\pi_1$, but allowing the likelihood (the conditional distribution of $\bYtil$ given $\tilde{\boldsymbol{h}}$) to be an arbitrary joint distribution. 
Defining the local-FDR (lfdr) function for the $i$th hypothesis in this model as
\begin{equation}\label{eq:lfdr}
\begin{aligned}
\operatorname{lfdr}_i(\by) := \PP(\htil_i = 0\lvert \bYtil=\by) = \frac{\sum_{\tau\in S_n}\mathbbm{1}(\theta_{\tau(i)} \in \Theta_0)f(\by; \btheta_\tau)}{\sum_{\tau\in S_n} f(\by; \btheta_\tau)}, \qquad  i=1,...,n, 
\end{aligned}
\end{equation}
HR21 show in their Theorem 2.1 that the optimal multiple testing procedure rejects all hypotheses $H_{0i}$ for which $\operatorname{lfdr}_i(\bYtil)\leq \hat{t}$, where the threshold $\hat{t}$ is possibly data-dependent. 
They go on to finding the threshold $\hat{t}$ explicitly, realizing that this amounts to solving an infinite integer program where both the objective and the constraint are linear, and show that $\hat{t}$ is indeed a function of $\bYtil$. 
In our two-groups model the $\htil_i$'s are exchangeable, and the joint likelihood is $\bYtil\lvert \tilde{\boldsymbol{h}} \sim g(\bytil\lvert \tilde{\boldsymbol{h}}) = \sum_{\tau\in S_1\times S_0} f(\bytil; \btheta_\tau)$, where for each $\tilde{\boldsymbol{h}}$ we denote by $S_1\times S_0$ the set of permutations $\tau$ of $[n]$ such that $S_1 = \{i:\htil_i=1\}$ is mapped to itself and $S_0 = \{i:\htil_i=0\}$ is mapped to itself. 
Because the model is different, the expression for $\operatorname{lfdr}_i$ in \eqref{eq:lfdr} will not be the same as in HR21, nor do the $\operatorname{lfdr}_i$'s have the same dependence structure as in HR21. 
However, the proof of their Theorem 2.1 does not use the particular form of the function $\operatorname{lfdr}_i$, only its definition as $\PP(\htil_i = 0\lvert \bYtil)$, and does not involve the joint distribution of the $\operatorname{lfdr}_i$'s, so  the statement of the theorem actually continues to hold in our two-groups model. 
Moreover, the proof of Theorem 2.1 in HR21 in fact shows that rejecting all $H_{0i}$ with $\operatorname{lfdr}_i(\bYtil)\leq t$, for any fixed $t$, is the optimal procedure {\em conditionally} on $\bYtil$ at the conditional FDR level corresponding to $t$; this fact is relevant to our discussion in Section \ref{subsec:conditional}. 
To conclude, the oracle PI rule in our problem must have the form
$$
\delta^*_{PI, i}(\bY) = 
\begin{cases}
1,&\text{if}\ \operatorname{lfdr}_i(\bY) \leq \hat{t}\\
0,&\text{otherwise}
\end{cases},
$$
where $\hat{t}$ is possibly data-dependent. 
To find $\hat{t} = \hat{t}(\bY)$ explicitly is a matter for future investigation. 
A natural route will be to try extending the approach in HR21 to the case of exchangeable $\htil_i$'s; if this is successful, implementation is another question, and we suspect that calculating $\hat{t}$ exactly will anyhow be computationally costly due to the particular distribution of $\bthetatil$. 
To conclude this example, we emphasize that the results of HR21 apply in a genuine two-groups model, i.e.~when the $h_i$ are random to begin with. 
Here we have applied a basic result from HR21 to characterize the oracle PI testing procedure that controls the FDR---alternatively, the attainable upper bound on power of any PI procedure that controls the FDR---in a strictly {\em frequentist} problem. 
We also point out that the result here is different from the optimal procedure in \cite{rosset2022optimal}, where the setup entails a simple alternative in each coordinate, and a procedure with strong control of the FDR is sought. 

\item {\em Confidence intervals for selected parameters}. 
It is easy to verify that for $K(\by, \btheta, \ba) = \operatorname{fcp}(\by, \btheta, \ba)$ and the choice of $L(\by, \btheta, \ba)$ specified in Section \ref{sec:unified}, the problem is PI if the selection rule $s:\mathcal{Y}^n\to \{0,1\}^n$ is PI, meaning that $s(\by_\tau) = (s(\by))_\tau$ for every permutation $\tau$ and any $\by\in \calY^n$; we assume this throughout. 
The modification suggested above applies, so that by \eqref{eq:optimization-mod}, the oracle PI rule $\delta^*_{PI}$ solves
\begin{equation}\label{eq:optimization-mod-fcr}
\begin{aligned}
\displaystyle{\minimize_{\delta}} \quad &\EE L(\bthetatil, \delta(\bYtil)) \\
\textup{s.t.} \quad &\EE \operatorname{fcp}(\bthetatil, \delta(\bYtil))\leq \alpha, 
\end{aligned}
\end{equation}
where $L(\bthetatil, \delta(\bYtil)) = {\sum_{i=1}^n s_i(\bYtil) \mu(\delta_i(\bYtil))}\big/{\sum_{i=1}^n s_i(\bYtil)}$. 
Formally, the oracle PI rule in this case is degenerate: remembering that $\bthetatil$ is discrete, if we take $\delta_i(\by)\equiv \{\theta_j: 1\leq j\leq n\}$, the support of $\thetatil_i$, then, trivially, this is a PI rule with $\operatorname{FCR}(\btheta, \delta)=0$ and $R(\btheta, \delta^*)=0$, since $\{\theta_j: 1\leq j\leq n\}$ has Lebesgue measure zero. 
To obtain a meaningful problem, we first state a result in a separate but related Bayesian framework, replacing the actual (discrete) distribution of $\bthetatil$ with some continuous distribution, then explain its relevance to the problem under consideration. 

\begin{proposition}\label{prop:fcr}
Consider minimizing \eqref{eq:optimization-mod} but now with $\bthetatil$ and $\bYtil$ jointly distributed s.t.~$\bthetatil\sim g(\bthetatil), \bYtil|\bthetatil\sim f(\cdot; \bthetatil)$ where $g$ is some Lebesgue density on $\RR^n$, and $f(\cdot; \bthetatil), \btheta\in \bTheta$, is the original parametric family \eqref{eq:model}. 
Then the minimizer, say $\delta^g$, has the property that $\delta^g_i(\by)$ is a highest posterior density (h.p.d.) region of some level for $\thetatil_i$. 
\end{proposition}

\begin{proof}
Let $\delta$ be a PI rule such that $\delta_i(\bY)\subseteq \RR$ is a set estimate for $\theta_i$ (with $\delta_i(\bY)=\emptyset$ if $s_i(\bY)=0$). 
We will show that $\delta$ can be improved by a procedure that constructs h.p.d.~regions w.r.t.~$g$. 
Fix any $\by\in \calY^n$ s.t.~$s_i(\by)=1$ for at least one $i$. 
For all $i$ s.t.~$s_i(\by)=1$, define 
$$
\begin{aligned}
\alpha_i(\by) := \PP_g(\thetatil_i\notin \delta_i(\by)\lvert \bYtil = \by) =
 \frac{\int \mathbbm{1}(\theta_i\notin \delta_i(\by)) g(\btheta)f(\by;\btheta) d\btheta}{\int g(\btheta)f(\by;\btheta) d\btheta}, 
\end{aligned}
$$
the posterior probability that $\delta_i(\by)$ does not cover $\thetatil_i$, and for each $i$ s.t.~$s_i(\by)=0$ set $\alpha_i(\by) = 0$ (an arbitrary value). 
Denote $\alpha(\by) := (\alpha_1(\by),...,\alpha_n(\by))$. 
Now take $\delta'$ to be the procedure that constructs for each selected parameter h.p.d.~region of level $1-\alpha_i(\by)$ for $\thetatil_i$, namely, 
$\delta'_i(\by) = \{\theta\in \Theta: g_i(\theta | \by)>t(\by)\}$ if $s_i(\by)=1$, where $g_i(\theta | \by)$ is the posterior density of $\thetatil_i$ given $\bYtil=\by$ evaluated at $\theta$, and where $t(\by)$ is s.t.~$\PP_g(\thetatil_i\notin \delta'_i(\by)\lvert \bYtil=\by)$. 
This is possible since the posterior of $\thetatil_i$ is continuous. 
Then 
\begin{align*}
\EE_g \big[ \operatorname{fcp}(\bYtil, \bthetatil, \delta'(\bYtil)) \big| \bYtil=\by \big]  
&=\EE_g \left[ \frac{\sum_{i=1}^n s_i(\bYtil) \mathbbm{1}(\thetatil_i \notin \delta'(\bYtil))}{\sum_{i=1}^n s_i(\bYtil)} \bigg| \bYtil\!=\by \right] \\[1ex]
&=\frac{\sum_{i=1}^n \! s_i(\by) \PP_g(\thetatil_i \notin \delta'(\by)\lvert \bYtil \! =\by)}{\sum_{i=1}^n s_i(\by)} = 
\frac{\sum_{i=1}^n \! s_i(\by) \alpha_i(\by)}{\sum_{i=1}^n s_i(\by)}, 
\end{align*}
which, by the same calculation applied to $\delta$, equals $\EE_g \big[ \operatorname{fcp}(\bYtil, \bthetatil, \delta(\bYtil)) \big| \bYtil=\by \big]$. 
This implies 
$
\EE_g \operatorname{fcp}(\bthetatil, \delta'(\bYtil)) = \EE_g \operatorname{fcp}(\bthetatil, \delta(\bYtil))
$
by taking expectation over $\bYtil$. 
Meanwhile, $\mu(\delta'_i(\by)) \leq \mu(\delta_i(\by))$ for every $\by$, because $\delta'_i(\by)$ is a h.p.d.~region of the same posterior level as $\delta_i(\by)$. 
Therefore $\EE_g L(\bthetatil, \delta'(\bYtil) \leq \EE_g L(\bthetatil, \delta(\bYtil)$. 
In conclusion, $\delta'$ has the same expected FCP as $\delta$ but smaller expected loss. 
\end{proof}


The proposition above does not apply directly to the original problem because the actual distribution of $\bthetatil$ is not continuous, but it is still has relevant implications. 
As remarked before, the oracle PI rule in the original problem formally has risk zero, which is not useful. 
By contrast, if $n$ is large and the empirical distribution of the fixed $\theta_i$'s is `dense', so that that the actual distribution of $\bthetatil$ (uniform on all permutations of $\btheta$) is well approximated by a continuous distribution with density $g^*_n$, Proposition \ref{prop:fcr} applied to $g^*_n$ already characterizes a nontrivial rule, whose risk can be used informally as a benchmark. 
Perhaps more importantly, the proposition above is informative from the algorithmic perspective, discussed in later sections, because it suggests that, in order to construct a good procedure, we should threshold {\em estimates} of $\PP(\thetatil_i \lvert \bYtil = \by)$. 
In producing such estimates it is common to model $\thetatil_i$ as realizations from an (unknown) continuous distribution, \citep[e.g.][]{efron2016empirical}, in which case the posterior is also continuous and Proposition \ref{prop:fcr} applies. 

\item {\em Effect-size estimation for selected parameters}. 
As in the previous item, we assume that the selection rule $s:\mathcal{Y}^n\to \{0,1\}^n$ is PI, in which case the entire problem is PI. 
Because $\calD_0$ here is unrestricted, i.e.~it includes all estimators $\delta:\calY^n\to \bTheta$, Corollary \ref{cor:unrestricted} applies. 
The selective posterior expected loss is 
$$
\begin{aligned}
\rho(\by,\ba) = \EE \big[\sum_{i=1}^n s_i(\by)(a_i - \thetatil_i)^2 \lvert \bYtil = \by \big] =\sum_{\{i:\ s_i(\by)=1\}} \!\!\!\! \EE \big[ (a_i - \thetatil_i)^2 \lvert \bYtil = \by \big]. 
\end{aligned}
$$
and the oracle PI rule is the minimizer of $\rho(\by,\ba)$ in $\ba$, namely, for every $i$ with $s_i(\by)=1$, 
\begin{equation}\label{eq:ol-PI-selective-estimation}
\delta_{PI, i}^*(\by) = \EE(\thetatil_i\lvert \bYtil=\by)
= \frac{\sum_{\tau\in S_n } \theta_{\tau(i)} f(\by; \btheta_{\tau})}{\sum_{\tau\in S_n }  f(\by; \btheta_{\tau})}
\end{equation}
(and the coordinates $i$ for which $s_i(\by)=0$ can be set arbitrarily). 

Another example that fits our framework is (a finite-alphabet version of) that considered in \cite{kontorovich2024distribution}. 
Here $\bY \sim \text{Multinomial}(N; \btheta)$, and of interest is to estimate $\theta_{I^*}, I^* \in \argmax_{i\leq n}Y_i$ (it is possible that the maximum count is attained by two or more $Y_i$'s). 
In the formulation from Section \ref{sec:unified} this corresponds to a selection rule that sets $s_i(\bY) = 1$ if $Y_i = \max(Y_1,..,Y_n)$ and $s_i(\bY) = 0$ otherwise. 
This is a PI problem for any choice of a loss function for the selected parameter $\theta_{I^*}$, e.g.~we could take the log-loss $-\theta_{I^*}\log \delta(\bY) - (1-\theta_{I^*})\log(1-\delta(\bY))$. 
Our general result says that the optimal PI estimator $\delta^*_{PI}(\by)$ is the minimizer, in $a$, of the selective posterior expected loss $\rho(\by,a)$. 
We emphasize that $i^* = I^*(\by)$  is deterministic here (because we condition on $\by$). 
\end{enumerate}

\section{Asymptotic equivalence}\label{sec:attainability}

Section \ref{subsec:comp} offered an approach for bypassing the computational cost of the bound \eqref{eq:bound-PI} using a deterministic asymptotic approximation for permanents of doubly stochastic matrices. 
The current section offers another direction for asymptotically approximating the greatest lower bound in some PI problems, by extending results on convergence of the risk of a weaker but computationally tractable oracle to $R(\btheta, \delta^*_{PI})$. 
For compound decision problems, as reviewed in Section \ref{sec:intro}, previous work has already established asymptotic equivalence of the risks of the PI and simple oracles. 
We now extends this by showing that in some PI problems {\em beyond} compound decision, this asymptotic equivalence continues to hold. 
Specifically, we consider a PI problem of estimating the (single) parameter corresponding to the maximum observation $Y_i$ under squared loss. 

Assume that the observed data follows the model \eqref{eq:model-cd}, with the parameters $\theta_i\in \Theta$ fixed, and let $f_i\coloneqq f(\cdot; \theta_i)$. 
We denote by $F_i \coloneqq F_{\theta_i}$
 the corresponding cumulative distribution functions. 
As in the previous sections, $\bY = (Y_1,...,Y_n)$, $\btheta = (\theta_1,...,\theta_n)$. 
Let $Y^*\coloneqq \max_{i\leq n}Y_i$ be the maximum of the observations. 
Also, let $I^* = I^*(\bY) \coloneqq \argmax_{i\leq n}Y_i$ the (random) index achieving the maximum, and let $i^*= i^*(\by)$ be its realized value. 
Now consider the problem of estimating $\theta_{I^*}$, the parameter of the maximum of the $Y_i$'s.  
Under squared loss, for any estimator $\delta(\bY)$ 
of $\theta_{I^*}$ the risk is 
\begin{equation*}
R_n(\btheta, \delta) \coloneqq \EE_{\btheta}(\delta(\bY)-\theta_{I^*})^2. 
\end{equation*}
The problem is PI per our earlier definitions, so the general theory developed above applies. 
Let $(\bthetatil, \bYtil)$ be jointly distributed according to 
\begin{equation}\label{eq:bayes}
\begin{aligned}
\tau \sim \text{uniform on }S_n, \qquad
(\thetatil_i, \Ytil_i)\lvert \tau 
\overset{ind}{\sim} (\theta_{\tau(i)}, Y_{\tau(i)}), 
\end{aligned}
\end{equation}
which is \eqref{eq:bayes-PI} under the independent model \eqref{eq:model-cd}. 
Define $\Ytil^*\coloneqq \max_{i\leq n}\Ytil_i$ and $\tilde{I}^*\coloneqq \argmax_{i\leq n}\Ytil_i$, the counterparts of $Y^*$ and $I^*$. 
Then the oracle PI estimator is 
$$
\delta_{PI}^*(\by) = \EE(\thetatil_{i^*}|\bYtil=\by). 
$$
Define also the oracle {\em simple} estimator by 
$$
\delta^*_{S}(\by) \coloneqq \EE(\thetatil_{i^*} \lvert \Ytil_{\tilde{I}^*} = y_{i^*}). 
$$
Note that $\delta_{PI}^*$ is the Bayes rule when the vector $(\tau(1),...,\allowbreak \tau(n))$ is uniformly distributed over the $n!$ permutations of $(1,...,n)$, whereas $\delta^*_{S}$ is the Bayes rule when $\tau(1),...,\tau(n)$ are i.i.d.~uniform draws from $\{1,...,n\}$. 
Indeed, in the latter case the conditional distribution of $\thetatil_{i^*}$ given $\bYtil=\by$ is the same as the conditional distribution of $\thetatil_{i^*}$ given $\Ytil_{i^*}=y_{i^*}$ since the pairs $(\tau(i), Y_{\tau(i)}), i\leq n$, are independent. 

We make the following assumptions. 
For some constant $C<\infty$:
\begin{enumerate}[start=1,label={(\bfseries A\arabic*)}]
\item It holds that $\max_{i\leq n} |\theta_i|<C$, and 
$$
\max_{i,j\leq n}\EE_{\theta_i}( (f(Y_i; \theta_j))^2/(f(Y_i; \theta_i))^2 )<C. 
$$ 
Also, there is some constant $\gamma>0$ for which $\min_{i,j\leq n}\PP_{\theta_i}( f(Y_i; \theta_j)/f(Y_i; \theta_i) > \gamma) \geq 1/2$ \label{assump1}
\item Defining 
$$
p_j(Y_i)\coloneqq \frac{f_j(Y_i)}{\sum_{k=1}^n f_k(Y_i)},\quad  i,j=1,...,n, 
$$
there is a constant $r>1$ such that, for all $i=1,...,n$, 
$$
\begin{aligned}
&\EE_{\theta_i} \bigg(\sum_{j=1}^n (p_j(Y_i))^r\bigg)^2 < \frac{C}{n^r},  \!\! \qquad \!\! \EE_{\theta_i} \bigg( \frac{1}{n\min_k p_k(Y_i)}\bigg)^r < C, \qquad \!\! \EE_{\theta_i} \bigg( 
\frac{\sum_{j=1}^n (p_j(Y_i))^2}{n\min_k p_k(Y_i)}
\bigg)^r < \frac{C}{n^r}
\end{aligned}
$$
\label{assump2}

\item for any fixed $m$, there is a constant $\tilde{C}$ and some $n_0$, s.t.~for all $n\geq n_0$, 
$$
\EE_{\btheta} 
\bigg(\frac{1}{\min_{k\leq n} F_k(Y^*)}\bigg)^m < \tilde{C}. 
$$
\label{assump3}
\end{enumerate}

The assumptions above are adaptations of the ones in \cite[][GR09 henceforth]{greenshtein2009asymptotic} for the selective estimation probelm, and quantify a sense in which the $\theta_i$, as well as their corresponding distributions, are not too far from each other. 
Assumption \ref{assump1} is exactly Assumption (G1) in Theorem 5.1 of GR09. 
Assumption \ref{assump2} is a counterpart of (G2) in GR09, which is consistent with requiring that the terms in \ref{assump2}, without the outer squares, are bounded in expectation. 
Assumption \ref{assump3} is an additional condition we need when the parameter of interest is chosen post-hoc. 

The assumptions \ref{assump1}-\ref{assump3} are satisfied for the bounded normal model, $Y_i\sim \calN(\theta_i, 1)$, independent, $i=1,...,n$, with $|\theta_i|\leq A<\infty$. 
For example, a calculation in GR09 shows that under this model, 
$$
\sum_{j=1}^n (p_j(Y_i))^2 = 
 \frac{\sum_{j=1}^n f_j^2(Y_i)}{(\sum_{k=1}^n f_k(Y_i))^2} \leq \frac{1}{n}e^{4A|Y_i|}. 
$$
Therefore, 
$$
\begin{aligned}
\EE_{\theta_i} \bigg[ \bigg( \sum_{j=1}^n (p_j(Y_i))^2 \bigg)^r \bigg] \leq \frac{1}{n^r} \EE_{\theta_i}e^{4Ar|Y_i|} 
\leq \frac{2}{n^r}e^{A^2(8r^2+4r)}, 
\end{aligned}
$$
using the fact that 
$$
\EE e^{t|Y_i|} = e^{t^2/2 + \theta_i t} + e^{t^2/2 - \theta_i t}\leq 2e^{t^2/2 + |\theta_i t|}
$$ 
for $Y_i\sim \calN(\theta_i, 1)$. 
The bounds on the other terms in \ref{assump2} can be shown with similar calculations. 
For (A3), 
$$
\begin{aligned}
\EE_{\btheta} \bigg( \frac{1}{(\min_k F_k(Y^*))^m } 1(Y^*< t)\bigg) \leq 
 \EE \bigg( \frac{1}{\Phi^m(X^*-A) }1(X^*< t)\bigg)
\end{aligned}
$$
for any fixed $t<0$, 
where $X^*\coloneqq \max_{i\leq n}X_i$ for $X_i\sim \calN(-A, 1)$, i.i.d. 
The right hand side is finite for large enough $n$, which can be shown, e.g., by writing out the integral with the density $f^*(u)= n\Phi^{n-1}(u+A)\phi(u+A)$ of $X^*$, and invoking elementary normal tail bounds, $z(z^2+1)^{-1}\phi(z)\leq \Phi(-z)\leq z^{-1}\phi(z)$ for $z>0$. 

The following theorem implies that the risks of the simple oracle and the PI oracle are asymptotically equal under the assumptions stated above. 
Our proof, included in the Appendix, 
is an adaptation of the proof of Theorem 5.1 in GR09, which asserts that the difference between $\EE_{\btheta} \sum_i (\hat{\theta}_i^{S} - \theta_i)^2$ and $\EE_{\btheta} \sum_i (\hat{\theta}_i^{PI} - \theta_i)^2$, where $\hat{\theta}_i^{S} \coloneqq \EE(\thetatil_i | \Ytil_i)$ and $\hat{\theta}_i^{PI} \coloneqq \EE(\thetatil_i | \bYtil)$ for a fixed $i$, is $O(1)$ under their assumptions. 
We note that from GR09's Theorem 5.1 it follows trivially that the difference in risks in our problem 
is $O(1)$, since the squared error for $\theta_{I^*}$ is bounded above by the sum of the errors for all $\theta_i$'s. 
We are interested in proving that the difference $\to 0$, which requires nontrivial modifications. 

\begin{theorem}\label{thm:convergence}
Under assumptions \ref{assump1}-\ref{assump3}, we have 
$$
R_n(\btheta, \delta^*_{S}) - R_n(\btheta, \delta^*_{PI}) = O(n^{-(r-1)/r}), 
$$
where $r$ is the constant from Assumption \ref{assump2}. 
\end{theorem}


Because in the bounded means model \ref{assump2} is satisfied for any $r>1$, we have 

\begin{corollary}
In the bounded normal model,  $Y_i\sim \calN(\theta_i, 1)$, independent, $i=1,...,n$, with $|\theta_i|\leq A<\infty$, we have 
$$
R_n(\btheta, \delta^*_{S}) - R_n(\btheta, \delta^*_{PI}) = O(n^{-\alpha})
$$
for any $0<\alpha<1$. 
\end{corollary}

\section{Nuisance Parameters}\label{sec:nuisance}

Before concluding we informally discuss an extension to the case where the model involves a nuisance parameter. 
Generalizing \eqref{eq:model}, suppose 
\begin{equation}\label{eq:model-nuisance}
\bY\sim f(\cdot; \btheta, \psi), 
\end{equation}
where $\btheta \in \bTheta\subseteq \Theta^n$, the unknown parameter of interest, and $\psi\in \Psi$, an unknown nuisance parameter, index a distribution on $\calY^n$. 
For simplicity, we think of $\psi$ here as one-dimensional, but the discussion that follows applies more generally when $\psi$ is $p$-dimensional with $p$ fixed (independent of $n$). 
We can easily extend the first part of Definition \ref{def:pi} to apply to this more general situation, by saying that the model \eqref{eq:model-nuisance} is PI, if it is PI per Definition \ref{def:pi} when $\psi = \psi_0$ is known. 
For example, suppose $\bY\sim \calN_n(\btheta, \sigma^2\Sigma)$ with $\Sigma$ known, so the parameter is $(\btheta, \sigma^2)\in \RR^n\times \RR_+$. 
To check if the model is PI, consider $\sigma^2 = \sigma^2_0$ as known, so now the model is formally $\bY\sim \calN_n(\btheta, \sigma^2_0\Sigma), \btheta\in \RR^n$. 
For any permutation $\tau\in S_n$ we have $\bY_\tau\sim f(\btheta_\tau; \sigma^2_0\Sigma_{\tau})$ where $(\Sigma_\tau)_{ij} := \Sigma_{\tau(i),\tau(j)}$. 
Thus, if, for example, $\Sigma_{ii} = 1$ and $\Sigma_{ij} = \rho$ for $i\neq j$, the model with known $\sigma^2 = \sigma^2_0$ is PI per Definition \ref{def:pi}. 
The conditions in Definition \ref{def:pi} on the loss and the feasible set $\calD_0$ are also checked when treating $\psi=\psi_0$ as known. 
For any PI problem in this extended sense, any decision rule $\delta$ (which is not a function of $\psi$) is similarly called PI if it is PI per the original definition when $\psi=\psi_0$ is known. 
The oracle PI rule is defined $\delta_{PI}^* := \argmin_{\delta\in \calD_0^{PI}} R((\btheta, \psi), \delta)$, i.e., it is the same oracle as in the known $\psi$ case, only now the rule $\delta_{PI}^*$ itself depends on $(\btheta, \psi)$ instead of just $\btheta$. 
In analogy to \eqref{eq:bound-PI}, the greatest lower bound on the risk of a PI rule is $R((\btheta, \psi), \delta^*_{PI})$, which, again, is now a function of both $\btheta$ and $\psi$. 
To summarize, the definition of a PI problem, a PI rule and the oracle PI rule extend with no essential modifications to the case with nuisance parameters (basically, the extended definition requires permutation invariance w.r.t.~$\btheta$ only). 
Thus, Sections \ref{sec:intro}-\ref{sec:attainability} are easily extendable to the case of nuisance parameters; this should not be surprising because, from the oracle perspective, $\psi$ is known. 

More interesting is the question of {\em uniform} attainability of the bound in the presence of a nuisance parameter. 
If we limit attention to the independent model, $Y_i\stackrel{ind}{\sim} f(\cdot; \theta_i, \psi)$, then, inspired by the results from compound decision theory and their extension in Section \ref{sec:attainability} for the case of no nuisance parameters, the task of asymptotically competing with the oracle would now amount to estimating $(G,\psi)$ under the assumption $\theta_i\sim G$, i.i.d., for an unknown, arbitrary $G$. 
But this model may not even be identifiable, for example if $Y_i\stackrel{ind}{\sim} \calN(\theta_i, \sigma^2)$ with $\theta_i\sim G$, there exist different pairs $(G,\sigma)$ yielding the same distribution on $\bY$ (even if it is known that $\sigma>0$). 
However, if we consider a model with 
replicates, 
\begin{equation}\label{eq:model-rep}
Y_{ij}\stackrel{ind}{\sim} f(\cdot; \theta_i, \psi),\ \ \ \ \ \ \ \ i=1,...,n,\ \ \ \ j=1,...,K, 
\end{equation}
where $K$ is fixed (independent of $n$), 
then $(G, \psi)$ is now obviously identifiable under the model that postulates $\theta_i\sim G$, i.i.d. 
The likelihood for the individual coordinates $(Y_{11},...,Y_{1K}, ..., Y_{n1}..., Y_{nK})$ is not PI, but the likelihood for the `blocks' $(\bY_1,...,\bY_n)$, where
$
\bY_{i} := (Y_{i1},...,Y_{iK}), 
$
is PI. 
Moreover, results in \cite{kiefer1956consistency} generally imply that, if we assume \eqref{eq:model-rep} and $\theta_i\sim G$, i.i.d., then $G$ and $\psi$ can be {\em jointly}  consistently estimated by maximum likelihood, 
$$
(\hat{G}, \hat{\psi}) = \argmax_{G, \psi} \prod_{i=1}^n f_{G, \psi}(Y_i), 
$$
where 
$$
f_{G, \psi}(\cdot):= \int f(\cdot;\theta, \psi)G(d\theta), 
$$
and the maximization in $G$ is over all possible distributions on $\Theta$. 
Under suitable conditions, then, we expect the plug-in EB estimator $\delta^{\hat{G},\hat{\psi}}$ to be consistent for $\delta^{G,\psi}$. 
In other words, the results of Kiefer and Wolfowitz suggest that, under \eqref{eq:model-rep}, we do not need to know $\psi$ in order to estimate the oracle simple rule $\delta_S^{G,\psi}$ consistently. 

\section{Discussion}\label{sec:discussion}

Compound decision theory, originating in the work of Robbins, is unusually appealing in its conclusiveness regarding the availability of optimal solutions. 
Indeed, in any compound decision problem, (i) restricting attention to PI rules is natural; 
(ii) the oracle PI rule, i.e.~the minimizer of the risk in that class, is the Bayes rule under a uniform prior on all permutations of $\btheta$; 
(iii) the risk of the oracle PI rule is asymptotically equal to the risk of the oracle simple rule, which replaces the aforementioned prior by an i.i.d.~prior with the same marginals, and, in addition, 
(iv) the risk of the oracle simple (and hence also of the oracle PI) rule is asymptotically attainable uniformly in $\btheta$ by an empirical Bayes rule. 
The current work extends some of the properties above to a much broader class of PI problems, where the target parameter is even allowed to depend on the data: our investigation begins with observing that (i) above naturally extends to this class of PI problems; we then extend (ii) to any such PI problem, which identifies explicitly the  ultimate lower bound on the risk for any fixed $n$; and, further, we extend (iii) to a non-compound PI problem, which  has important implications on computation of the bound---and hence its usefulness---for large $n$.


Proving that also property (iv) extends beyond compound decision, would add much appeal from the algorithmic viewpoint, as it would demonstrate that the bound is asymptotically attainable uniformly in $\btheta$ by a legal, empirical Bayes procedure. 
This would establish that Robbins’s EB approach is asymptotically instance-optimal not only in compound decision problems (for which it was originally developed) but also in some PI problems. 
Such a result is not presented in the current article, and deserves separate consideration because the situation is much more subtle when moving from compound decision to the general PI class considered here: whereas for compound decision problems the specific structure of the loss in  \eqref{eq:loss-sum}, in particular the normalizing factor $n^{-1}$, enables to obtain  convergence results, there is much more flexibility in the loss function when considering a general PI problem. 
For example, the loss considered in the selective estimation problem of Section \ref{sec:attainability} is markedly different from a compound loss. 
Thus, the feasibility of attaining the simple oracle risk with an EB decision rule is expected to depend strongly on the particular loss under consideration. 

Other remaining challenges include finding good deterministic approximations to \eqref{eq:bound-PI}, 
especially since our findings in Section \ref{subsec:comp} show discrepancy between the results from MCMC approximations and the proposed approximation based on Mccullagh's formula. 
The issues with using the latter, reported in Section \ref{subsec:comp}, impede calculation of the required lower bound, but perhaps a more accurate version of Mccullagh's formula (for example, if higher-order terms can be obtained) can alleviate the problem. 
Approximation of the bound beyond the independent model \eqref{eq:model-cd} is a further challenge. 
An important direction for future work is checking existence of and conditions for the asymptotic convergence of Section \ref{sec:attainability} in other PI problems. 
Of course, this would also have implications on approximations of the bound \eqref{eq:bound-PI}, because when asymptotic convergence holds, it is enough to use the (easily computable) risk of the simple oracle for sufficiently large $n$. 
Another interesting direction to investigate is whether some ideas of the theory developed here extend to problems that are not PI. 
A notable example is inference in linear model, $\bY = \bX\boldsymbol{\beta} + \boldsymbol{\epsilon}$, say with $\bX$ a fixed $n\times p$ matrix and $\boldsymbol{\epsilon}$ a normally distributed noise term. 
In various statistical tasks (estimation, variable selection, etc.), it is common to employ statistical procedures that first compute an importance statistic $T_j(\bY)$ for each of the $p$ columns of $\bX$, then process $T_1(\bY),...,T_p(\bY)$ in a way that depends only on their values, not the labels $1,...,p$. 
Thus, while the {\em model} for the statistics $T_j(\bY)$ will generally not be PI, when choosing a procedure we often (voluntarily) restrict ourselves to a PI {\em rule}. 
This observation can lead to a meaningful lower bound by associating each value of $\boldsymbol{\beta}$ with the oracle PI decision rule, i.e., the rule that minimizes the risk at $\boldsymbol{\beta}$ among all rules that ignore the labels. 

\section*{Acknowledgments}
This work was supported by the Israeli Science Foundation (ISF) under grant no.~2679/24. 
I am indebted to Daniel Xiang for insightful discussions, and especially for educating me on the ideas underlying Section \ref{subsec:comp}. 
I am also  grateful to Eitan Greenshtein for his comments on an earlier version of this draft.

\begin{appendix}
\section{Proofs}\label{appx:proofs}

\begin{proof}[Proof of Theorem \ref{thm:main}]
To simplify notation, instead of introducing the random vectors $(\bthetatil, \bYtil)$, we will distinguish between the original (frequentist) model for $\bY$, 
\begin{equation}\label{eq:appx-model}
Y_i\overset{ind}{\sim} f_i,\ \ \  \ \ \ \ \  i=1,...,n, 
\end{equation}
and the two-level (Bayes) model for $\bY$, 
\begin{equation}\label{eq:appx-bayes}
\tau \sim \text{uniform on }S_n;\ \ \ \ \ \ \ \ \ \ \ Y_i|\tau \overset{ind}{\sim} f_{\tau(i)},\ i=1,...,n, 
\end{equation}
by using a subscript $\btheta$ on operations taken under \eqref{eq:appx-model}, e.g.~$\EE_{\btheta}$, and using no subscript at all for operations taken under \eqref{eq:appx-bayes}. 
Notice that \eqref{eq:appx-bayes} and \eqref{eq:bayes-PI} are equivalent, except that in the former $\bthetatil = \btheta_{\tau}$ is not defined explicitly. 
Correspondingly, the PI oracle is 
$$
\delta^*_{PI}(\by) = \EE(\theta_{\tau(i^*)}|\bY=\by), 
$$
and the simple oracle is 
$$
\delta^*_{S}(\by) = \EE(\theta_{\tau(i^*)} | Y_{I^*}= y_{i^*}), 
$$
remembering that $i^* = i^*(\by)$. 
In what follows, $\bY_{-i}$ denotes the vector obtained from $\bY$ by dropping $Y_i$. 

First, as in GR09, we note that the analysis can be done under the Bayes model \eqref{eq:appx-bayes}, in particular we can replace the risks by the Bayes risks, because both oracles considered are permutation invariant. 
Moreover, using the fact that $\EE(\EE[(\theta_{\tau(I^*)}|\bY) - \theta_{\tau(I^*)})| \bY]=0$, 
\begin{align*}
R_n(\btheta, \delta^*_S) = \EE (\delta^*_S(\bY) - \theta_{\tau(I^*)})^2 
&= \EE \EE [(\delta^*_S(\bY) - \EE(\theta_{\tau(I^*)}|\bY) +  \EE(\theta_{\tau(I^*)}|\bY) - \theta_{\tau(I^*)})^2 | \bY] \\
&= \EE \EE [(\delta^*_S(\bY) - \EE(\theta_{\tau(I^*)}|\bY))^2 + (\EE(\theta_{\tau(I^*)}|\bY) - \theta_{\tau(I^*)})^2 | \bY] \\
&= \EE [(\delta^*_S(\bY) \! -  \! \EE(\theta_{\tau(I^*)}|\bY))^2  \! +  \! (\EE(\theta_{\tau(I^*)}|\bY)  \! - \theta_{\tau(I^*)})^2 ] \\
&= \EE (\delta^*_S(\bY) - \delta^*_{PI}(\bY))^2 + R_n(\btheta, \delta^*_{PI}), 
\end{align*}
hence to show the claim is equivalent to showing $\EE (\delta^*_S(\bY) - \delta^*_{PI}(\bY))^2 =  O(n^{-(r-1)/r})$. 
Write the PI and simple oracles explicitly, 
$$
\begin{aligned}
\delta^S(\bY) \coloneqq \EE(\theta_{\tau(I^*)}|Y_{I^*}) =\sum_j\theta_j\PP(\tau(I^*) = j|Y_{I^*}) 
 = \sum_j\theta_j\PP(\tau(I^*) = j|Y_j, I^*=j)
\end{aligned}
$$
and 
$$
\begin{aligned}
\delta^{PI}(\bY) \coloneqq \EE(\theta_{\tau(I^*)}|\bY) &= \sum_j\theta_j\PP(\tau(I^*) = j|\bY) 
= \sum_j\theta_j\PP(\tau(I^*) = j|Y_j, I^*=j) \\
&= \sum_j\theta_j\PP(\tau(I^*) = j|Y_j, I^*=j) \frac{f(\bY_{-j}|Y_j, I^*=j, \tau(j)=j)}{f(\bY_{-j}|Y_j, I^*=j)}. 
\end{aligned}
$$
Let 
$$
p_j^*(Y_j)\coloneqq \PP(\tau(j)=j|Y_j, I^*=j)
$$
and denote by 
$$
\gjzerostar(\bY_{-j}|Y_j) \coloneqq f(\bY_{-j}|Y_j, I^*=j, \tau(j)=j)
$$
the density, under \eqref{eq:appx-bayes}, of $\bY_{-j}$ conditional on $Y_j$, $I^*=j$ and $\tau(j)=j$, and by
$$
g_1^*(\bY_{-j}|Y_j) \coloneqq f(\bY_{-j}|Y_j, I^*=j)
$$
the density under \eqref{eq:appx-bayes} of $\bY_{-j}$ conditional on $Y_j$ and $I^*=j$ (note that the latter does not depend on $j$, only on $Y_j$, because of exchangeability under \eqref{eq:appx-bayes}). 

Putting 
$$
W_j^*(\bY_{-j}, Y_j) \coloneqq \frac{\gjzerostar(\bY_{-j}|Y_j)}{g_1^*(\bY_{-j}|Y_j)}, 
$$
we then have 
$$
\delta^*_S(\bY) - \delta^*_{PI}(\bY) = \sum_{j=1}^n \theta_j p_j^*(Y_j) \bigg( W_j^*(\bY_{-j}, Y_j) - 1 \bigg). 
$$
Define two other densities on $\bY_{-j}$, conditional on $Y_j$ and $I^*=j$, by 
\begin{align*}
&\gjzerostartilde(\bY_{-j}|Y_j) \coloneqq \prod_{i\neq j} f_i^*(Y_i)\\
&\gjonestartilde(\bY_{-j}|Y_j) \coloneqq \gjzerostartilde(\bY_{-j}|Y_j) \bigg( \sum_{k\neq j} p_k^*(Y_j)\frac{f_j^*(Y_j)}{f_k^*(Y_j)} + p_j^*(Y_j) \bigg), 
\end{align*}
where 
$$
f_i^*(Y_i) \coloneqq f_{i}(Y_i|Y_j, i^*=j) = \frac{f_i(Y_i)}{F_i(Y_j)} \bfone(Y_i\leq Y_j)
$$
is the conditional density, under the original model \eqref{eq:appx-model}, of $Y_i$ given $Y_j$ and that the maximum is achieved by the $j$th coordinate. 
Note that $\gjzerostar$ and $g_1^*$ are obtained from $\gjzerostartilde$ and $\gjonestartilde$, respectively, by applying a random permutation (to $\bY_{-j}$). 
Hence, invoking Lemma 2.1 from GR09, we have
$$
\begin{aligned}
\EE \big(
|W_j^*(\bYnotj, Y_j) - 1|^2 \big| Y_j, i^*=j
\big) 
\leq  
\EE_{\gjonestartilde} \big(
\frac{\gjzerostartilde}{\gjonestartilde} - 1
\big)^2 =
\EE_{\gjonestartilde} \big(
\frac{\gjzerostartilde}{\gjonestartilde} - 2 + \frac{\gjonestartilde}{\gjzerostartilde}
\big).
\end{aligned}
$$
Let 
\begin{align*}
&L^* \coloneqq \frac{\gjonestartilde}{\gjzerostartilde} = \sum_{k\neq j} p_k^*(Y_j) \frac{f_j^*}{f_k^*}(Y_k) + p_j^*(Y_j)\\
&V^* \coloneqq \frac{n}{4} \gamma^* \min_k p_k^*(Y_j),\ \ \ \ \ \ \ \gamma^*\coloneqq \gamma \min_k F_k(Y_j), 
\end{align*}
where $\gamma$ is a constant as in \ref{assump1}. 
Note that, for $C$ as in \ref{assump1}, 
\begin{equation}\label{eq:obs1}
\begin{aligned}
\EE_{\theta_k} \bigg( \frac{(f_j^*(Y_k))^2}{(f_k^*(Y_k))^2} \bigg| Y_j, Y_k\leq Y_j \bigg) 
&= \frac{(F_k(Y_j))^2}{(F_j(Y_j))^2} \EE_{\theta_k} \bigg( \frac{(f_j(Y_k))^2}{(f_k(Y_k))^2} \bigg| Y_j, Y_k\leq Y_j \bigg)\\
&= \frac{(F_k(Y_j))^2}{(F_j(Y_j))^2} \int \frac{f_j^2}{f_k^2}(y) \frac{f_k(y)}{F_k(Y_j)}\bfone(y\leq Y_j) dy\\
&\leq \frac{1}{(\min_k F_k(Y_j))^3} \int \frac{f_j^2(y)}{f_k^2(y)}f_k(y) dy\\
&= \frac{1}{(\min_k F_k(Y_j))^3} \EE_{\theta_k} \bigg[ \frac{f^2_j(Y_k)}{f^2_k(Y_k)}\bigg] \leq C \frac{1}{(\min_k F_k(Y_j))^3}
\end{aligned}
\end{equation}
and that 
\begin{equation}\label{eq:obs2}
\begin{aligned}
&\PP_{\theta_k} \bigg( \frac{f_j^*(Y_k)}{f_k^*(Y_k)} >\gamma^* \bigg| Y_j, Y_k\leq Y_j \bigg) \\
&= \PP_{\theta_k} \bigg( \frac{f_j}{f_k} \geq \gamma^* \frac{F_j}{F_k}(Y_j) \bigg| Y_j, Y_k\leq Y_j \bigg) \\
&= \int \frac{f_k(y)}{F_k(y)}\bfone(y\leq Y_j) \bfone\bigg( \frac{f_j(y)}{f_k(y)} \geq  \gamma^* \frac{F_j}{F_k}(Y_j) \bigg) dy\\
&\geq \!\!  \int \!\! f_k(y) \bfone(y\leq Y_j) \bfone\bigg( \! \frac{f_j(y)}{f_k(y)} \geq  \gamma^* \frac{1}{ \min_k F_k(Y_j) }(Y_j) \! \bigg) dy\\
&\geq  \int \!\! f_k(y) \bfone\big( \frac{f_j}{f_k}(y) \geq \gamma \big) dy\\
&= \PP_{\theta_k} \big( \frac{f_j}{f_k}(Y_k) \geq \gamma \big) \geq 1/2. 
\end{aligned}
\end{equation}
Then
\begin{align}
\EE\big(
|W_j^*(\bYnotj, Y_j) - 1|^2 \big| Y_j, i^*=j
\big)
&\leq 
\EE_{\gjonestartilde} \big(
\frac{1}{L^*} - 2 + L^*
\big)=\EE_{\gjonestartilde} 
\frac{(L^*-1)^2}{L^*} \notag \\
&\leq \frac{1}{V^*} \EE_{\gjonestartilde} (L^*-1)^2 \bfone(V^*>L^*) + \EE_{\gjonestartilde}\frac{\bfone(V^*\leq L^*)}{L^*}. \label{eq:sum}
\end{align}
For the first term in the sum, using \eqref{eq:obs1}, 
\begin{equation}\label{eq:summand1}
\begin{aligned}
&\frac{1}{V^*} \EE_{\gjonestartilde} (L^*-1)^2 \bfone(V^*>L^*) 
\leq \frac{1}{V^*} \EE_{\gjonestartilde} (L^*-1)^2 \\
&=\textnormal{Var}_{\gjonestartilde}(L^*) 
=\textnormal{Var}_{\gjonestartilde} \bigg( \sum_{k\neq j} p_k^*(Y_j) \frac{f_j^*}{f_k^*}(Y_k) + p_j^*(Y_j) \bigg)\\
&= \sum_k (p_k^*(Y_j))^2 
\EE_{\theta_k} \bigg( \frac{(f_j^*(Y_k))^2}{(f_k^*(Y_k))^2} \bigg| Y_j, Y_k\leq Y_j \bigg)\\
&\leq C \frac{1}{(\min_k F_k(Y_j))^3} \sum_k (p_k^*(Y_j))^2. 
\end{aligned}
\end{equation}
For the second term in the sum in \eqref{eq:sum}, bound
$$
\begin{aligned}
L^* \geq \gamma^* \min_k p_k^*(Y_j) \sum_{k=1}^n \bfone \bigg( \frac{f_j^*}{f_k^*}(Y_k) \geq \gamma^* \bigg)  \succeq \gamma^* \min_k p_k^*(Y_j)(1+U)
\end{aligned}
$$
where $U\sim Bin(n-1,1/2)$. 
By the same exact calculation as in GR09, 
\begin{equation}\label{eq:summand2}
\begin{aligned}
\EE_{\gjonestartilde}\frac{\bfone(V^*\leq L^*)}{L^*} 
\leq O(e^{-n}) \frac{1}{\gamma^* n \min_k p_k^*(Y_j)}. 
\end{aligned}
\end{equation}
From \eqref{eq:sum}, \eqref{eq:summand1} and \eqref{eq:summand2}, 
\begin{align}\label{eq:bound}
\EE\big(
|W_j^*(\bYnotj, Y_j) - 1|^2 \big| Y_j, i^*=j
\big) \leq  C\frac{\sum_k (p_k^*(Y^*))^2}{(\min_k F_k(Y^*))^3}  + O(e^{-n}) \frac{1}{\gamma^* n \min_k p_k^*(Y^*)}. 
\end{align}
Now, we can rewrite 
\begin{align*}
p_k^*(Y_j) = \PP(\pi(j)=k | Y_j, i^*=j) 
&=\frac{\PP(\pi(j)=k) f_k(Y_j)\prod_{i\neq k}F_i(Y_j)}{
\sum_l \PP(\pi(j)=l) f_l(Y_j)\prod_{i\neq l}F_i(Y_j)
}\\
&\qquad\qquad = \frac{f_k(Y_j)}{\sum_l f_l(Y_j)} \frac{\sum_l f_l(Y_j)}{\sum_l \frac{F_k}{F_l}(Y_j)f_l(Y_j)}
= p_k(Y_j)v_k(Y_j)
\end{align*}
where
$$
v_k(Y_j) \coloneqq \frac{\sum_l f_l(Y_j)}{\sum_l \frac{F_k}{F_l}(Y_j)f_l(Y_j)}. 
$$
Noting that
$$
\begin{aligned}
\min_k p_k^*(Y_j) = \min_k p_k(Y_j) v_k(Y_j) \geq \min_k p_k(Y_j)  \min_k v_k(Y_j), 
\end{aligned}
$$
and observing that 
$$
\min_k F_k(Y_j)\leq v_k(Y_j) \leq (\min_k F_k(Y_j))^{-1}, 
$$
we have 
\begin{align*}
\frac{\sum_k (p_k^*(Y_j))^2}{V^*} &\leq 4 \frac{\sum_k (p_k^*(Y_j))^2(v_k(Y_j))^2}{\gamma \min_kF_k(Y_j) n \min_k p_k (Y_j) \min_k v_k(Y_j)} \\[1ex]
&\leq 
\frac{4}{(\min_kF_k(Y_j))^4}
\frac{\max_k v_k^2(Y_j)}{\min_k v_k^2(Y_j)}
\frac{\sum_k (p_k^*(Y_j))^2}{n\gamma \min_k p_k(Y_j)} \\[1ex]
&\leq 
\frac{4}{(\min_kF_k(Y_j))^4}
\frac{\sum_k (p_k^*(Y_j))^2}{n\gamma \min_k p_k(Y_j)}
\end{align*}
and
$$
\begin{aligned}
\frac{1}{\gamma^* n \min_k p_k^*(Y_j)}  &\leq \frac{1}{\gamma \min_k F_k(Y_j) \min_k p_k(Y_j)  \min_k v_k(Y_j)}\\[1ex]
&\quad \leq 
\frac{1}{(\min_k F_k(Y_j))^2}\frac{1}{\gamma n \min_k p_k(Y_j)}. 
\end{aligned}
$$
Then, noting that 
$$
\begin{aligned}
\EE\big(|W_j^*(\bYnotj, Y_j) - 1|^2 \big| Y_j, i^*=j\big) = \EE\big(|W_j^*(\bY_{-i^*}, Y^*) - 1|^2 \big| Y_{i^*}\big), 
\end{aligned}
$$ 
we can further bound \eqref{eq:bound} as 
\begin{equation}\label{eq:mean-cond-ystar}
\begin{aligned}
\EE\big(|W_j^*(\bY_{-i^*}, Y^*) - 1|^2 \big| Y^*\big)& \leq 
\frac{4C}{(\min_k F_k(Y^*))^7} \frac{\sum_k (p_k(Y^*))^2}{n\gamma \min_k p_k(Y^*)} \\
&\quad+ O(e^{-n}) \frac{1}{(\min_k F_k(Y^*))^2} \frac{1}{\gamma n \min_k p_k(Y^*)}. 
\end{aligned}
\end{equation}
Now take expectation (over $Y$) on both sides, and note that we can assume the fixed-$\btheta$ model, because the distribution of $Y^*$ is the same under \eqref{eq:model} and \eqref{eq:bayes}. 
For the first term on the right hand side of \eqref{eq:mean-cond-ystar}, invoking H\"{o}lder's inequality with $r>1$ as in (A2), we have that for $n$ large enough, 
$$
\begin{aligned}
\EE \frac{4C}{(\min_k F_k(Y^*))^7} \frac{\sum_k (p_k(Y^*))^2}{n\gamma \min_k p_k(Y^*)}
\leq
 \bigg[ \EE \bigg( \frac{4C}{(\min_k F_k(Y^*))^7} \bigg)^\frac{r}{r-1} \bigg]^{\frac{r-1}{r}}
&\bigg[ \EE \bigg( 
\frac{\sum_k (p_k(Y^*))^2}{n\gamma \min_k p_k(Y^*)}
\bigg)^r \bigg]^{1/r}\\[1.5ex]
&\qquad \leq \text{const}\cdot n^{-\frac{r-1}{r}}, 
\end{aligned}
$$
using the assumptions and the trivial fact
$$
\begin{aligned}
\EE \bigg( 
\frac{\sum_k (p_k(Y^*))^2}{n\gamma \min_k p_k(Y^*)}
\bigg)^r
\leq 
\sum_{i=1}^n
\EE \bigg( 
\frac{\sum_k (p_k(Y_i))^2}{n\gamma \min_k p_k(Y_i)}
\bigg)^r \leq n\cdot \frac{C}{n^r}. 
\end{aligned}
$$
The expectation of the second term on the right hand side of \eqref{eq:mean-cond-ystar} is $o(e^{-n})$, hence we conclude that for large enough $n$, 
$$
\EE\big(|W_j^*(\bY_{-i^*}, Y^*) - 1|^2 \big) \leq \text{const}\cdot n^{-\frac{r-1}{r}}. 
$$
Finally, 
\begin{align*}
\EE (\delta^S(\bY) - \delta^{PI}(\bY))^2 &= \EE \EE [(\delta^S(\bY) - \delta^{PI}(\bY))^2|Y^*]\\
&= \EE \EE [ (\sum_{j=1}^n \theta_j p_j^*(Y_j) \big( W_j^*(\bY_{-j}, Y_j) - 1 \big))^2|Y^* ]\\
&\leq \max_j|\theta_j|^2 \EE\sum_j p_j^*(Y^*) \EE\big(|W_j^*(\bY_{-i^*}, Y^*) - 1|^2 \big| Y^*\big)\\
&\leq \max_j|\theta_j|^2 \EE [\max_j \EE\big(|W_j^*(\bY_{-i^*}, Y^*) - 1|^2 \big| Y^*\big) \sum_j p_j^*(Y^*)]\\
&\leq \text{const} \cdot n^{-\frac{r-1}{r}}
\end{align*}
for $n$ large enough, where in the last inequality we used the fact that $\sum_jp_j^*(t)\equiv 1$. 

\end{proof}

\end{appendix}

\bibliographystyle{abbrvnat}
\bibliography{refs}

\end{document}